\documentclass[11pt,a4paper]{scrartcl}

\usepackage{a4wide}
\usepackage{latexsym}
\usepackage[USenglish]{babel}
\usepackage[utf8]{inputenc}
\usepackage[T1]{fontenc}
\usepackage{hyperref}
\usepackage{amssymb}

\usepackage[amsthm,thmmarks]{ntheorem}

\usepackage{mathrsfs}
\usepackage{graphicx}
\usepackage{amsmath}
\usepackage{mathdots}
\usepackage{amsfonts}
\usepackage{amsxtra}
\usepackage{url}
\usepackage{mathptmx}
\usepackage{helvet}
\usepackage{stmaryrd}
\usepackage{dsfont}
\usepackage{setspace}
\usepackage{enumerate}
\usepackage{MathDef}
\usepackage{verbatim}
\usepackage{color}
\usepackage{cleveref}
\usepackage{cancel}

\def\dif{\mathrm{d}}

\def\DD{\mathcal{D}}

\def\HH{\mathcal{H}}
\def\II{\mathcal{I}}

\def\LL{\mathcal{L}}
\def\MM{\mathcal{M}}
\def\NN{\mathcal{N}}

\def\YY{\mathcal{Y}}

\def\Pas{\P\text{-a.s.}}

\renewcommand{\P}{\ensuremath{\mathbb{P}}}
\DeclareMathOperator*{\argmin}{arg\,min}
\newcommand{\vt}{\ensuremath{\vartheta}}
\newcommand{\he}{\ensuremath{\widehat{\epsilon}}}
\DeclareMathOperator{\vecc}{vec}
\DeclareMathOperator{\tr}{tr}

\theoremstyle{plain}
\newtheorem{theorem}{Theorem}[section]
\newtheorem{lemma}[theorem]{Lemma}
\newtheorem{proposition}[theorem]{Proposition}
\newtheorem{definition}[theorem]{Definition}
\newtheorem{corollary}[theorem]{Corollary}
\newtheorem{assumptionletter}{{\textbf{Assumption}}}

\theoremstyle{definition}

\newtheorem{remark}[theorem]{Remark}

\numberwithin{equation}{section}
\allowdisplaybreaks

\renewcommand{\labelenumi}{(\roman{enumi})}
\usepackage{titlesec}
\titleformat{\paragraph}{\normalfont\bfseries}{\theparagraph}{}{}
\titlespacing*{\paragraph}{0pt}{3.25ex plus 1ex minus .2ex}{0.1em}
\title{Information Criteria for \\[2mm] Multivariate CARMA Processes}

\author{Vicky Fasen \setcounter{footnote}{1}\thanks{Institute of Stochastics, Englerstra{\ss}e 2,
D-76131 Karlsruhe, Germany. \emph{Email:}
\href{mailto:vicky.fasen@kit.edu}{vicky.fasen@kit.edu}}
\thanks{Financial
support by the Deutsche Forschungsgemeinschaft through the research
grant FA 809/2-2.} \and Sebastian Kimmig \setcounter{footnote}{0}\thanks{Institute of Stochastics, Englerstra{\ss}e 2,
D-76131 Karlsruhe, Germany. \emph{Email:}
\href{mailto:sebastian.kimmig@kit.edu}{sebastian.kimmig@kit.edu}}}

\date{}

\begin{document}
%
%%%%%%%%%%%%%%%%%  Titel und Autoreninformationen  %%%%%%%%%%%%%%%%%%%%%%%%%%%%%%%%
\maketitle
%
%%%%%%%%%%%%%%%%%  Abstract und AMS Classification  %%%%%%%%%%%%%%%%%%%%%%%%%%%%%%%
\begin{abstract}
Multivariate continuous-time ARMA$(p,q)$ (MCARMA$(p,q)$) processes
are the \linebreak continuous-time
analog of the well-known vector ARMA$(p,q)$ processes. They have attracted interest
over the last years. %In the literature, there exist methods to estimate
%the parameters of a CARMA$(p,q)$ process. However, these methods usually assume that the orders $p$ and $q$ are known.
Methods to estimate the parameters of an MCARMA process  require an identifiable parametrization  such as the Echelon form
with a fixed Kronecker index, which is in the
one-dimensional case the degree $p$ of the autoregressive polynomial. Thus, the Kronecker index
has to be known in advance before the parameter estimation is done.
When this is not the case information criteria can be used to estimate the
Kronecker index and the degrees $(p,q)$, respectively.
In this paper we investigate information criteria for MCARMA processes based on quasi maximum likelihood estimation. Therefore, we first derive the asymptotic
properties of quasi maximum likelihood estimators for MCARMA processes in a misspecified parameter space. Then, we present necessary and sufficient conditions for
information criteria to be strongly and weakly consistent, respectively.
In particular, we study the well-known Akaike Information Criterion (AIC) and the Bayesian Information Criterion
(BIC) as special cases.
\end{abstract}

\noindent
\begin{tabbing}
\emph{AMS Subject Classification 2010: }\=Primary:  62B10,  62F12, 62M86
\\ \> Secondary:  62F10, 62M10
\end{tabbing}

\vspace{0.2cm}\noindent\emph{Keywords:} AIC, BIC,  CARMA process, consistency, information criteria, law of iterated logarithm,  Kronecker index, quasi maximum likelihood estimation

%%%%%%%%%%%%%%%%%  Einleitung  %%%%%%%%%%%%%%%%%%%%%%%%%%%%%%%%%%%%%%%%%%%%%%%%%
\section{Introduction}
%
%
%%%%%%%%%%%%%%%%%  Literaturverzeichnis  %%%%%%%%%%%%%%%%%%%%%%%%%%%

% \cite{boubacar}

In this paper we study necessary and sufficient conditions for weak and strong consistency
of information criteria for multivariate  continuous-time ARMA$(p,q)$ (MCARMA$(p,q)$) processes.
One-dimensional Gaussian CARMA
processes were already investigated by Doob \cite{Doob1944} in 1944
and L\'{e}vy-driven CARMA processes were propagated at the beginning of this century by Peter Brockwell,  see \cite{Brockwell:2014} for an overview.
An $\R^s$-valued L\'{e}vy process $(L(t))_{t\geq 0}$ is a stochastic process in $\R^s$ with independent and
stationary increments, $L(0)=0_s$ $\p$-a.s. and c\`{a}dl\`{a}g (continue \`{a} droite, limite \`{a} gauche) sample paths. Special cases of L\'{e}vy processes are Brownian motions and (compound) Poisson processes. Further information on L\'{e}vy processes can be found in \cite{applebaum,bertoin,sato}, for example.
A formal definition of an MCARMA process was given recently in \cite{marquardtstelzer}; see \Cref{Section:2} of this paper. The idea behind it
is that for a two-sided $\R^s$-valued L\'{e}vy process
 $L=(L(t))_{t\in\R}$, i.e. $L(t)=L(t)\1_{\{t\geq 0\}}-\wt L(t-)\1_{\{t<0\}}$ where $(\wt L(t))_{t\geq 0}$ is an independent
copy of the L\'{e}vy process $(L(t))_{t\geq 0}$,   and positive integers $p>q$,
a $d$-dimensional MCARMA$(p,q)$ process is the solution to the
stochastic differential equation
\begin{eqnarray} \label{eq1.1}
     {P}(D)Y(t)={Q}(D)D L(t) \quad \mbox{ for } t\in\R,
\end{eqnarray}
where $D$ is the differential operator,
\beam \label{Pol P}
    {P}(z):= I_{d\times d}z^p+A_1z^{p-1}+\ldots+A_{p-1}z+A_p
\eeam
with $A_1,\ldots, A_p\in \R^{d\times d}$ is the autoregressive polynomial and
\beam \label{Pol Q}
    {Q}(z):=B_0z^q+ B_1z^{q-1}+\ldots+ B_{q-1}z+B_q
\eeam
with $B_0,\ldots,B_q\in \R^{d\times s}$
is the moving average polynomial. There are a few papers studying the statistical inference
of MCARMA processes, e.g. \cite{brockwellschlemm,fasen:2014b,fasen:2014,Ferrazzano:Fuchs,schlemmstelzer2,schlemmstelzer}.
In particular, \cite{schlemmstelzer} derive the asymptotic behavior of the quasi maximum likelihood estimator (QMLE)
under the assumption that the underlying parameter space $\Theta$ with $N(\Theta)$ parameters contains the true parameter and satisfies
some identifiability assumptions; see \cite{Brockwell:Davis:Yang:2011} as well. These are
typical assumptions for estimation procedures. For a one-dimensional CARMA process we only obtain identifiability  when the degree $p$ of the
autoregressive polynomial is fixed for all processes generated by parameters in the parameter space; in the multivariate setup the
Kronecker index, which specifies in detail the order of the coefficients of the multivariate
autoregressive polynomial, has to be fixed. If we know the Kronecker index
we know the degree $p$ of the autoregressive polynomial as well.
But if we observe data, how do we know what is the true Kronecker index of the data, so that we do the parameter estimation
in a suitable parameter space $\Theta$? That is the point where we require model selection criteria or, synonymously, information criteria.
The most prominent model selection criteria are the Akaike Information Criterion (AIC)
introduced in  \cite{akaike} by Akaike, the Schwarz Information Criterion (SIC), also known as BIC (Bayesian Information Criterion), going back to
\cite{Schwarz:1978}, and the Hannan-Quinn criterion in \cite{Hannan:Quinn:1979}. The AIC approximates the Kullback-Leibler discrepancy,
whereas the BIC approximates the Bayesian a posteriori distribution of the different candidate models. The Hannan-Quinn criterion
is based on the AIC of Akaike but with a different penalty term to obtain a strongly consistent information criterion.
Information criteria for multivariate ARMAX processes and their statistical inference are well-studied in the
monograph~\cite{hannandeistler}; see also \cite{brockwell} for an overview of model selection criteria for  ARMA processes.
An extension of the AIC to multivariate weak ARMA processes is given in \cite{boubacar}. There exist only
a few papers investigating information criteria independent of the underlying model, e.g.
\cite{sinwhite} present very general likelihood-based information
criteria and their properties, and \cite{cavanaughneath} derive the BIC.
All of these information criteria have in common that they are likelihood-based and choose  as candidate
model the model for which the information criterion attains the lowest value.
They are of the form
\begin{equation*}
\text{IC}_n(\Theta_{}) :=\wh \LL ( \widehat{\vt}^n, Y^n) + N(\Theta) \frac{C(n)}{n}.
\end{equation*}
In our setup $Y^n=(Y(h),\ldots,Y(hn))$ is a sample of length $n$ from an MCARMA process, %$\Theta$ is a parameter space containing MCARMA processes which depend on $N(\Theta)$ parameters,
$\wh \LL$ is the properly normalized quasi log-likelihood function, $\wh{\vt}^n$ is the QMLE and $C(n)$ is a penalty term. % Under suitable conditions, we will then have $IC_n(\Theta)\stackrel{\p}{\to}\mathcal{Q}(\vartheta^*)$, where $\mathcal{Q}$ is a function that attains its minimum in $\vt_0$.
We choose the parameter space as the most suitable for which the information criterion is lowest, this means that for two parameter
spaces $\Theta_1,\Theta_2$ we say that $\Theta_1$ fits the  data  better than $\Theta_2$ if  we have $IC_n(\Theta_1)<IC_n(\Theta_2)$.
A strongly consistent information criterion chooses  the correct space asymptotically with probability $1$, and for a weakly consistent information criterion the  convergence
 to the true space holds in probability.
The sequence  $C(n)$  can be interpreted as a penalty term for the inclusion of more parameters into the model. Without the penalty term, the criterion would always choose the model with more parameters if we compare two parameter spaces  both containing a parameter that generates the data.
However, this is not feasible, since the inclusion of too many parameters ultimately leads to an interpolation of the data, such that the model would not provide information about the process generating the data anymore. The employment of an information criterion can therefore be seen as seeking a trade-off between accuracy and complexity.
%However, there is no guarantee that an information criterion  actually selects the correct parameter space. One might hope that the probability of choosing
 %the correct parameter space increases if the number of observations increases.
% As it will turn out, the condition $C(n)/n\to 0$
%guarantees that underfitting is not possible, which means that asymptotically there is no positive probability  of choosing a parameter space which cannot
%generate the process underlying the data. However, $C(n)/n\to 0$ is not sufficient to exclude overfitting, which on the contrary means selecting a space with more parameters than necessary. We will introduce sufficient criteria to exclude this phenomena as well.

The rest of the paper is structured in the following way. In \Cref{Section:2} we present basic facts on MCARMA processes
and state space models. Since our information criteria are based on quasi maximum likelihood estimation
we define first, in \Cref{Section:3.1}, the quasi log-likelihood function for  MCARMA processes and in \Cref{Section:3.2} the model assumptions.
Then, in \Cref{Section:3.4}, we derive the asymptotic normality of the QMLE 
extending the results given in \cite{schlemmstelzer} to a misspecified parameter space. For the proof of strong consistency of the information criteria
we require some knowledge about the asymptotic behavior of the quasi log-likelihood function $\wh\LL$ as well. For this reason we
prove in \Cref{Section:3.5} a law of the iterated logarithm for the quasi log-likelihood function $\wh\LL$. \Cref{Section:4}
contains the main results of the paper:  necessary and sufficient conditions for strong and weak
consistency of information criteria. In particular, we investigate Gaussian MCARMA processes where
the results are explicit. Special information criteria are the AIC and the BIC which are the topic of \Cref{Section:5}.
Finally, we conclude with a simulation study in \Cref{sec:simulation}. The Appendix contains some auxiliary results. %\vspace*{-0.3cm}

\subsubsection*{Notation} \vspace*{-0.3cm}

We  use the notation $\stackrel{\DD}{\to}$ for weak convergence and
$\stackrel{\p}{\to}$ for convergence in probability.
For two random vectors $Z_1,\,Z_2$ the notation $Z_1 \stackrel{\DD}{=} Z_2$ means equality in distribution.
We use as norms the Euclidean norm $ \lVert\cdot \rVert$ in $\R^d$
and the spectral norm $ \lVert\cdot \rVert$ for
matrices, which is submultiplicative and induced by the Euclidean norm. Recall that two norms
on a finite-dimensional linear space are always equivalent and hence,
our results remain true if we replace the Euclidean norm by any other norm.
  The matrix $0_{d\times s}$
is the zero matrix in $\R^{d\times s}$ and $I_{d\times d}$ is the
identity matrix in $\R^{d\times d}$.
For a vector $x\in\R^d$ we write $x^T$ for its
transpose. For a matrix $A\in\R^{d\times d}$ we denote by $\tr(A)$ its trace, by $\det(A)$ its determinant and by $\lambda_{\text{max}}(A)$ its largest eigenvalue. If $A$ is symmetric and positive semidefinite we write $A^{\frac12}$ for the principal square root, i. e. $A^{\frac12}$ is the symmetric, positive semidefinite matrix satisfying $A^\frac12 A^\frac12 = A$. For two matrices $A \in \R^{d \times s}$ and $B \in \R^{r \times n}$, we denote by $A \otimes B$ the Kronecker product, which is an element of $\R^{dr \times sn}$. The notation $\vecc(A)$ describes the $ds \times 1$ row vector which results from stacking the columns of $A$ beneath each other.
The symbols $\E$, $\Var$, and $\Cov$ stand for the
expectation, variance and covariance operators, respectively. For a sequence of random variables $(X_n)_{n \in \N}$ we say that $X_n$ is $o_{\text{a.s.}}(a_n)$ if $\left| {X_n}/{a_n} \right| \to 0$ as $n \to \infty$  $\Pas$ and likewise that $X_n$ is $O_{\text{a.s.}}(a_n)$ if $\limsup_{n \to \infty} \left| {X_n}/{a_n} \right| < \infty$  $\Pas$
We write $\partial_i$ for the partial derivative
operator with respect to the $i$-th coordinate and $\nabla=(\partial_1,\ldots,\partial_r)$
for the gradient operator in $\R^r$. Finally, by $\partial^2_{i,j}$ we denote the second partial derivative with respect to the coordinates $i$ and $j$, and by $\nabla^2_\vt f$ we denote the Hessian matrix of the function $f$. When there
is no ambiguity, we use $\partial_i f(\vartheta_0)$, $\nabla_\vartheta f(\vartheta_0)$ and $\nabla^2_\vt f(\vt_0)$  as shorthands for
$\partial_i f(\vartheta)|_{\vartheta=\vartheta_0}$,  $\nabla_\vartheta f(\vartheta)|_{\vartheta=\vartheta_0}$ and $\nabla^2_\vt f(\vt)|_{\vt=\vt_0}$, respectively.
We interpret $ \nabla_\vartheta f(\vartheta)$ as a column vector.
In general $C$ denotes a constant which may change from line to line.

\section{MCARMA processes and state space processes} \label{Section:2}

We start with the formal definition of an MCARMA process, which can be interpreted as
solution of \eqref{eq1.1}.

\begin{definition} \label{Definition CARMA}
Let $(L(t))_{t\in\R}$  be an $\R^s$-valued L\'{e}vy process with $\E\|L(1)\|^2<\infty$ and let
the polynomials ${P}(z),{Q}(z)$ be defined as in \eqref{Pol P} and \eqref{Pol Q} with $p,q\in\N_0$, $q<p$,
and $B_0\not=0_{d\times s}$.
Moreover, define
\beao
    A=\left(\begin{array}{ccccc}
        0_{d\times d} & I_{d\times d} & 0_{d\times d} & \cdots & 0_{d\times d}\\
        0_{d\times d} & 0_{d\times d} & I_{d\times d} & \ddots & \vdots \\
        \vdots & & \ddots & \ddots & 0_{d\times d}\\
        0_{d\times d} & \cdots & \cdots & 0_{d\times d} & I_{d\times d}\\
        -A_p & -A_{p-1} & \cdots & \cdots & -A_1
    \end{array}\right) \in \R^{pd\times pd},
\eeao
$C=(I_{d\times d},0_{d\times d},\ldots,0_{d\times d}) \in \R^{d\times pd}$ and $B=(\beta_1^T \cdots \beta_p^T)^T\in \R^{pd\times s}$ with
\beao
   \beta_1:=\ldots:=\beta_{p-q-1}:=0_{d\times s}\quad
   \mbox{ and } \quad \beta_{p-j}:=-\sum_{i=1}^{p-j-1}A_i \beta_{p-j-i}+B_{q-j},  \quad j=0,\ldots,q.
\eeao
Assume that the eigenvalues of $A$ have strictly negative real parts.
Then the $\R^{d}$-valued {\em causal \linebreak MCARMA$(p,q)$ process} $Y=(Y(t))_{t\in\R}$ is  defined by the state space
equation
\beam \label{CARMA:observation}
    Y(t)=C X(t) \quad \mbox{ for } t\in\R,
\eeam
where
%\beam \label{CARMA:state}
%    X(t)=\int_{-\infty}^{t}\e^{A(t-s)} B\, \dd L(s) \quad \mbox{ for } t\in\R
%\eeam
$X$ is the stationary unique solution to the $pd$-dimensional
stochastic differential equation
\begin{equation}\label{cstateeq}
\dd X(t)=A X(t)\,\dd t+B \,\dd L(t).
\end{equation}
\end{definition}
In particular, the MCARMA$(1,0)$ process and
$X$ in \eqref{cstateeq} are multivariate Ornstein-Uhlenbeck processes.
For more details on the well-definedness of the  MCARMA$(p,q)$ process see~\cite{marquardtstelzer}.
% Moreover, \cite[Lemma~3.8]{marquardtstelzer} says that the set of zeros of $\det(P(z))$ is equal to the set of eigenvalues of $A$, which means that
% for a  MCARMA$(p,q)$ process the eigenvalues of $A$ have strictly negative real parts.
The class of MCARMA processes is huge. Schlemm and Stelzer~\cite[Corollary~3.4]{schlemmstelzer}
showed that the class
of continuous-time state space models of the form
\begin{eqnarray} \label{2.1}
    Y(t)=CX(t) \quad\mbox{ and }\quad \dd X(t)=A X(t)\,\dd t+ B \,\dd L(t),
\end{eqnarray}
where $A\in\R^{N\times N}$ has only eigenvalues with strictly negative
real parts, $B\in\R^{N\times s}$ and $C\in\R^{d\times N}$,
and the class of causal MCARMA processes are equivalent if $\E\|L(1)\|^2<\infty$
and $\E[L(1)]=0_s$.
In general, when we talk about an MCARMA process or a state space model $Y$, respectively, corresponding
to $(A,B,C,L)$, we mean that the MCARMA process $Y$ is defined as in \eqref{2.1} and shortly
write $Y=\text{MCARMA}(A,B,C,L)$.

In this paper we observe the MCARMA process only on a discrete equidistant time-grid with grid
distance $h>0$. It is well-known that the Ornstein-Uhlenbeck process $(X(t))_{t\in\R}$ sampled at $h\Z$
is an AR(1)-process with
\begin{eqnarray*}
    X(kh)=\e^{Ah}X((k-1)h)+N_{h,k}, \quad k\in\Z,
\end{eqnarray*}
where $N_{h,k}=\int_{(k-1)h}^{kh}\e^{A(kh-t)}B\,\dd L(t)$ is a sequence of i.i.d. random vectors.
We denote its covariance matrix by $\Cov(N_{h,k})=\cancel{\Sigma}_h$. Hence, $(Y(kh))_{k\in\Z}$ is the output process of the discrete-time state space model
\begin{equation}  \label{eq2.5}
    Y(kh)=CX(kh) \quad \text{ where } \quad
    X(kh)=\e^{Ah}X((k-1)h)+N_{h,k}.
\end{equation}
This discrete-time state space representation is basic for quasi maximum likelihood estimation.

\section{Quasi maximum likelihood estimation }\label{sec:QMLEMCARMA} \label{Section:3}

\subsection{Definition} \label{Section:3.1}

Since the MCARMA process observed at discrete  equidistant time points is a discrete-time state space model as given in \eqref{eq2.5},
we use quasi maximum likelihood estimation for discrete-time state space models with respect to identification issues. We now review the most important aspects of estimation
 as it is done in \cite{schlemmstelzer} for MCARMA processes. The estimation is based on the Kalman filter, which calculates the linear
innovations of a Gaussian discrete-time state space model; originally introduced in \cite{kalman} and described in a time series context in \cite[§12.2]{brockwell}.

\begin{definition} \label{definition innovations}
Let $(Z_k)_{k \in \Z}$ be an $\R^d$-valued stationary stochastic process with finite second moments. The linear innovations $\epsilon = (\epsilon_k)_{k \in \Z}$  are then defined as $\epsilon_k = Z_k - P_{k-1} Z_k$, where $P_k$ denotes the orthogonal projection onto the space $\overline{\text{span}} \{ Z_j: - \infty < j \leq k \}$ and the closure is taken in $L^2$.
\end{definition}
Note that this definition ensures that the innovations of such a process are stationary, uncorrelated and have mean $0$.
In the following we calculate the linear innovations of $(Y(kh))_{k\in\Z}$.

For this purpose, let $\Omega$ be the solution to the discrete-time Riccati equation
\begin{align*}
\Omega^{} &= \e^{A h} \Omega^{}  \e^{A^T h} + {\cancel{\Sigma}_h}
- \left( \e^{A h} \Omega^{} C^T \right) \left( C \Omega^{} C^T \right)^{-1} \left( \e^{ A h} \Omega C^T \right)^T, \label{calcvhat1}
\end{align*}
which exists by  \cite[Proposition 2.1i)]{schlemmstelzer}. Then, the Kalman gain matrix is
\begin{equation*}\label{calcvhat2}
K^{} = \left( \e^{A h} \Omega^{} C^T \right) \left( C \Omega^{} C^T \right)^{-1}.
\end{equation*}
The linear innovations of $(Y(kh))_{k\in\Z}$ can be calculated as % the unique stationary solution to
\begin{equation}\label{kalman3}
\begin{split}
\epsilon_k = Y_k - C\widehat{X}_k, \quad k\in \Z,  \quad\text{ with } \quad
\widehat{X}_k = \sum_{j=1}^{\infty}(\e^{A h} - K C)^{j-1}K Y((k-j)h).
\end{split}
\end{equation}
The covariance matrix of the innovations is $V := \E[\epsilon_k\epsilon_k^T]=C \Omega^{} C^T$.
If we observe data we unfortunately do not know the model parameter behind it and hence,
we have to calculate the  so-called pseudo-innovations.
In the following we assume that our data set is generated by  a continuous-time state space
model $(A, B, C, L)$, i.e. $Y=\text{MCARMA}(A,B,C,L)$. Moreover, we have
have a parametric family of MCARMA models $(A_\vt,B_\vt,C_\vt,L_\vt)$  with $\vt$ in the
parameter space $\Theta \subset \R^{N(\Theta)}$, $N(\Theta) \in \N$. The aim is to find
$\vt_0\in\Theta$ such that $\text{MCARMA}(A_{\vt_0},B_{\vt_0},C_{\vt_0},L_{\vt_0})=Y$.
Therfore, we calculate for every $\vt\in\Theta$ the steady-state Kalman gain matrix $K^{}_\vt$ and covariance matrix $V^{}_\vt$ via the discrete-time Riccati equation
\begin{align*}
\Omega^{}_\vt &= \e^{A_\vt h} \Omega^{}_\vt  \e^{A^T_\vt h} + \cancel{\Sigma}_{\vt,h}
- \left( \e^{A_\vt h} \Omega^{}_\vt C^T_\vt \right) \left( C_\vt \Omega^{}_\vt C^T_\vt \right)^{-1} \left( \e^{ A_\vt h} \Omega^{}_\vt C^T_\vt \right)^T, \label{calcvhat1}
\end{align*}
as
\begin{equation*}
K^{}_\vt = \left( \e^{A_\vt h} \Omega^{}_\vt C^T_\vt \right) \left( C_\vt \Omega^{}_\vt C^T_\vt \right)^{-1} \quad \mbox{ and }\quad V^{}_\vt = C_\vt \Omega^{}_\vt C^T_\vt.
\end{equation*}
Based on this the \textsl{pseudo-innovations} are defined as
\begin{equation*}\label{approxpseudoinno}
\begin{split}
 %=(\e^{A_\vt h} - K_\vt C_\vt) \widehat{X}_{\vt, k-1} + K_\vt Y_{\vt_0}((k-1)h), \quad k \in\Z, \\
    \varepsilon_{\vt, k} = Y((k-1)h) - C_\vt \widehat{X}_{\vt, k}, \quad k \in \Z, \quad \text{ with } \quad
    \widehat{X}_{\vt, k} =\sum_{j=1}^{\infty}(\e^{A_\vt h} - K_\vt C_\vt)^{j-1}K_\vt Y((k-j)h).
\end{split}
\end{equation*}
Note that $(\widehat{X}_{\vt, k})_{k\in\Z}$ can also be calculated recursively by
\begin{eqnarray*}
    \widehat{X}_{\vt, k}=(\e^{A_\vt h} - K_\vt C_\vt) \widehat{X}_{\vt, k-1} + K_\vt Y((k-1)h).
\end{eqnarray*}
For $\vt_0$ such that $\text{MCARMA}(A_{\vt_0},B_{\vt_0},C_{\vt_0},L_{\vt_0})=Y$ the pseudo-innovations $(\epsilon_{\vt_0,k})_{k\in\N}$ are the innovations as given in \Cref{definition innovations} and $V_{\vt_0}=\E[\epsilon_{\vt_0,1}\epsilon_{\vt_0,1}^T]$. With this, $-2/n$ times the Gaussian log-likelihood of the model associated to $\vt$ is
%In this situation, \cite[Eq. (11.5.4)]{brockwell} \marginpar{\"{a}ndern} tells us that $-\frac{2}{n}$ times the logarithm of the Gaussian likelihood of $\vt$ can be written as
\begin{equation} \label{loglike}
\LL (\vartheta, Y^n) = \frac{1}{n}\sum_{k=1}^{n} (d \log(2 \pi) + \log( \text{det}( V_{\vartheta})) + \epsilon^T_{\vartheta, k} V^{-1}_{\vartheta} \epsilon_{\vartheta,k} ).
\end{equation}
Defining
\begin{equation*} % \label{loglike}
    l_{\vt, k}:= d \log(2 \pi) + \log( \text{det}( V_{\vartheta})) + \epsilon^T_{\vartheta, k} V^{-1}_{\vartheta} \epsilon_{\vartheta,k} , \quad k\in\Z,
\end{equation*}
we can also write
%\begin{equation*}
    $\LL (\vartheta, Y^n) = \frac{1}{n}\sum_{k=1}^{n}l_{\vt, k}.$
% Gaussian likelihood in this situation means that this is the exact log--likelihood function if the innovations are normally distributed. If this assumption fails to hold, we nevertheless work with the function $\LL$, hence the name quasi maximum likelihood estimation.
% Finally, we define
The expectation of this random variable is
 \begin{equation*}
 \mathscr{Q}(\vt): = \E \left[ \LL( \vt, Y^n ) \right]. %= \E \left[ l_{\vt,1} \right].
\end{equation*}
In practical scenarios it is not possible to calculate the pseudo-innovations, as they are defined in terms of the full history of the process $Y$ but we have
only finitely many observations. Suppose now that we have $n$ observations of the output process $Y$, contained in the sample \linebreak $Y^n = (Y(h), \ldots, Y(nh))$.  Therefore we need a method to approximate the pseudo-innovations based on this finite sample. We initialize the filter  at $k=1$ by prescribing $\widehat{X}_{\vt, 1} = \widehat{X}_{\vt, \text{initial}}$ and use the recursion
\begin{equation*}
\begin{split}
\widehat{X}_{\vt, k} &= (\e^{A_\vt h} - K_\vt C_\vt) \widehat{X}_{\vt, k-1} + K_\vt Y((k-1)h), \quad k \geq 2, \\
\he_{\vt, k} &= Y((k-1)h) - C_\vt \widehat{X}_{\vt, k}, \quad k \in \N.
\end{split}
\end{equation*}
%$\widehat{X}_{\vt, \text{initial}}$ can be sampled from the stationary distribution of $X_\vt$ if possible or simply set to some deterministic value.
The $\he_{\vt, k}$ are denoted as \textit{approximate pseudo-innovations}. Substituting the approximate pseudo-innovations for their theoretical counterparts in \eqref{loglike}, we obtain the quasi log-likelihood function
as
\begin{equation}\label{emploglike}
\widehat{\LL} (\vartheta, Y^n): = \frac{1}{n}\sum_{k=1}^{n} (d \log(2 \pi) + \log( \text{det}( V_{\vartheta})) + \he^T_{\vartheta, k} V^{-1}_{\vartheta} \he_{\vartheta,k} ).
\end{equation}
The QMLE based on the sample $Y^n$ is then given by
\begin{equation}\label{GQMLE}
\widehat{\vt}^n: = \argmin\limits_{\vt \in \Theta} \widehat{\LL} (\vartheta, Y^n).
\end{equation}
The idea is that $\widehat{\vt}^n$ is an estimator for the \textsl{pseudo-true parameter}
\begin{equation}\label{defpseudotrue}
\vt^* := \argmin_{\vt \in \Theta}  \mathscr{Q}(\vt).
\end{equation}
The function $\mathscr{Q}$ attains its minimum at $\vt^*$ in the space $\Theta$. However, if we minimize only over $\Theta$ and
$\Theta$ does not contain a parameter generating $Y$ then it is not clear that the minimum, and hence $\vt^*$, is uniquely defined.
On the other hand, if there is a $\vt_0\in\Theta$ with $\text{MCARMA}(A_{\vt_0},B_{\vt_0},C_{\vt_0},L_{\vt_0})=Y$ then
$\vt^*=\vt_0$. The last case was investigated in \cite{schlemmstelzer}.

\subsection{Assumptions} \label{Section:3.2}

In this section we give the model assumptions which we require for the asymptotic results on the QMLE $\wh\vartheta^n$.
The next definition introduces the concept of minimal algebraic realizations of matrix polynomials, which is essential in describing identifiable parametrizations of MCARMA processes.
\begin{definition}
Let $H$ be a $d \times s$ rational matrix function, i. e. a $d \times s$ matrix whose entries are rational functions of the variable $z \in \R$.
\begin{itemize}
 \item[(a)] A matrix triple $(A,B,C)$ is called an algebraic realization of $H$ of dimension $N$ if $H(z) = C (z I_{N \times N} - A)^{-1} B$ for every $z \in \R$, where $A \in \R^{N \times N}$, $B \in \R^{N \times s}$ and $C \in \R^{d \times N}$.
 \item[(b)] A minimal realization of $H$ is an algebraic realization of $H$ of dimension smaller or equal to the dimension of every other algebraic realization of $H$. The dimension of a minimal realization of $H$ is the McMillan degree of $H$.
\end{itemize}
\end{definition}
We now present the assumptions we use  in the development of the asymptotic theory of the QMLE: % in one place for easy reference:
\setcounter{assumptionletter}{1}
\begin{assumptionletter}\strut
         \label{as_D}
         \renewcommand{\theenumi}{B.\arabic{enumi}}
         \renewcommand{\labelenumi}{\theenumi}
         \begin{enumerate}
                 \item \label{as_D1} The parameter space $\Theta$ is a compact subset of $\R^{N(\Theta)}$.
                 \item \label{as_D2} For each $\vt \in \Theta$, it holds that $\E[L_\vt] = 0$, $\E \| L_\vt(1) \|^2 < \infty$ and the covariance matrix
                 \linebreak $\Sigma^L_\vt = \E \left[ L_\vt(1) L^T_\vt(1) \right]$ is non-singular.
                 \item \label{as_D3} For each $\vt \in \Theta$, the eigenvalues of $A_\vt$ have strictly negative real parts and
                 %\item \label{as_D7} For all $\vt \in \Theta$, the eigenvalues of $A_\vt$
                 are elements of $\{ z \in \C: -\frac{\pi}{h} < \text{Im}(z) < \frac{\pi}{h} \}$.
                 \item \label{as_D8} The pseudo-true parameter $\vt^*$ as defined in \eqref{defpseudotrue}
                 is an element of the interior of $\Theta$.
                 \item \label{as_D9} For the L\'{e}vy process $L$ which drives the observed process $Y$ there exists a positive number $\delta$ such that $\E \| L(1) \|^{4+\delta} < \infty$.
                 \item \label{as_D10} For every $\epsilon > 0$ there exists a
$\delta(\epsilon) > 0$ such that
\begin{equation*}
      \mathscr{Q}(\vt^\ast) \leq \min_{\vt \in B_\epsilon(\vt^*)^c \cap \Theta}   \mathscr{Q}(\vt)    - \delta(\epsilon),
\end{equation*}
where $B_\epsilon(\vt^*)$ is the open ball with center $\vt^*$ and radius $\epsilon$.
 \item \label{as_D11} The Fisher information matrix  of the QMLE is non-singular.
   \item \label{as_D4} The functions $\vt \mapsto A_\vt$, $\vt \mapsto B_\vt$, $\vt \mapsto C_\vt$ and $\vt \mapsto \Sigma^L_\vt$ are three times continuously differentiable. Moreover, for each $\vt \in \Theta$, the matrix $C_\vt$ has full rank.
                 \item \label{as_D5} For all $\vt \in \Theta$, the triple $(A_\vt, B_\vt, C_\vt)$ is minimal with  McMillan degree $N$.
                 \item \label{as_D6} The family of output processes  $(\text{MCARMA}(A_\vt, B_\vt, C_\vt, L_\vt))_{\vt \in \Theta}$ is identifiable from the spectral density.
         \end{enumerate}
\end{assumptionletter}

\begin{remark}\label{chap3remark}  $\mbox{}$\\
 (a) \, Every process in the family $(A_{\vt}, B_{\vt}, C_{\vt}, L_{\vt})_{\vt \in \Theta}$ has a different spectral density by $\ref{as_D6}$. Moreover, by $\ref{as_D5}$ it is also ensured  for two parameter spaces $\Theta$ and $\Theta'$ both satisfying \autoref{as_D} with different McMillan degrees that
     the processes generated by parameters in $\Theta$ are different from the processes generated by parameters in $ \Theta' $.\\
(b)\, Assumption $\ref{as_D10}$ is a property called identifiable uniqueness. It makes sure that $\vt^*$ is the unique minimum of $\mathscr{Q}(\vt)$ in
  $\Theta$ (cf. \cite[p. 28]{white1996}).
 In the correctly specified case, i. e. when the space $\Theta$ contains $\vt_0$ with $\text{MCARMA}(A_{\vt_0},B_{\vt_0},C_{\vt_0},L_{\vt_0})=Y$, the identifiable uniqueness follows from some properties satisfied by the innovations associated to the true parameter $\vt_0$, %see \cite[Lemma 2.9 and 2.10]{schlemmstelzer},
 i. e. Assumption $\ref{as_D10}$ can then be dropped without any replacement.\\
(c) \, In case of a correctly specified parameter space, we can replace Assumption $\ref{as_D11}$ by the assumption that there exists a positive index $i_0$ such that the $\left[ (i_0 +2) d^2 \right] \times r$ matrix
\begin{equation*} \nabla_\vt \begin{pmatrix} \left[ I_{(i_0+1) \times (i_0+1)} \otimes K^{T}_{\vt_0} \otimes C_{\vt_0} \right] \begin{pmatrix} \vecc \exp (I_{N \times N} h) \\ \vecc \exp(A_{\vt_0} h) \\ \vdots \\ \vecc \exp (A^{i_0}_\vt h) \end{pmatrix} \\ \vecc V_{\vt_0} \end{pmatrix} \end{equation*}
has rank $r$. This condition is used in \cite{schlemmstelzer} as Assumption C11 and guarantees the desired non-singularity.
\end{remark}

\begin{remark}\label{echelonremark}
An MCARMA process $(A,B,C,L)$ in Echelon form with Kronecker index \linebreak $m=(m_1,\ldots,m_d)$ has the property that
 $A = (A_{ij})_{i, j = 1, \ldots, d} \in \R^{N \times N}$ is a block matrix with blocks $A_{ij} \in \R^{m_i \times m_j}$ given by
$$A_{ij} = \begin{pmatrix} 0 & \ldots & \ldots & \ldots & \ldots & 0 \\ \vdots & & & & & \vdots \\ 0 & \ldots & \ldots & \ldots & \ldots & 0 \\ \alpha_{ij,1} & \ldots & \alpha_{ij, \min(m_i+\1_{\{i>j\}},m_j)} & 0 & \ldots & 0 \end{pmatrix} + \delta_{i,j} \begin{pmatrix} 0 & & \\ \vdots & I_{(m_i -1) \times (m_i - 1)} & \\ 0 &  & \\ 0 & \ldots & 0 \end{pmatrix}$$
and
$$C = \begin{pmatrix} 1 & 0 & \ldots & 0 & \vdots & 0 & 0 & \ldots & 0 & \vdots &   & \vdots &  &  &  &  \\
                                     &   &  &  &  &  &  &  &  &  &  &  &  &  & 0_{(d-1)  \times m_d} &   \\
                                     &    &         &    & \vdots & 1 & 0 & \ldots & 0 & \vdots &   & \vdots &   &   &   &   \\
                                     & 0_{(d-1) \times m_1} &  &  &  &  &  &  &  &  &  &  &  &  &  &   \\
                                     &    &         &    & \vdots &   & 0_{(d-2)  \times m_2} &   &   & \vdots &   & \vdots & 1 & 0 & \ldots & 0 \end{pmatrix}$$

The matrix $ B = (b_{ij}) \in \R^{N \times s}$ is unrestricted.
Moreover,  the polynomials $P(z)=[p_{ij}(z)]$ and $Q(z)=[q_{ij}(z)]$ are of the form
\begin{equation*}\label{echeloneq}
p_{ij}(z) = \delta_{i,j} z^{m_i} - \sum_{k=1}^{\min(m_i+\1_{\{i>j\}},m_j)} \alpha_{ij,k} z^{k-1} \quad \mbox{ and }\quad q_{ij}(z) = \sum_{k=1}^{m_i} \kappa_{\nu_1 + \ldots+ \nu_{i-1} + k, j} z^{k-1},
\end{equation*}
where $\kappa_{i,j}$ is the $(i,j)$th entry of the matrix $K = TB$, where $T = (T_{ij})_{i,j = 1, \ldots, d} \in \R^{N \times N}$ is a block matrix with blocks $T_{ij} \in \R^{m_i \times m_j}$ given by
$$ T_{ij} = \begin{pmatrix} -\alpha_{ij,2} & \ldots & - \alpha_{ij, \min(m_i+\1_{\{i>j\}},m_j)} & 0 & \ldots & 0 \\
                   \vdots &\iddots & & & & \vdots \\
                   -\alpha_{ij,\min(m_i+\1_{\{i>j\}},m_j)} & & & & & \vdots \\
                   0 & & & & & \vdots \\
                   \vdots & & & & & \vdots \\
                   0 & \ldots & \ldots & \ldots & \ldots & 0 \end{pmatrix} + \delta_{i,j} \begin{pmatrix} 0 & 0 & \ldots & \ldots & 0 & 1 \\
                                                                                                         0 & 0 & \ldots & & 1 & 0 \\
                                                                                                         \vdots & \vdots & &\iddots & &\vdots \\
                                                                                                         \vdots & &\iddots & & \vdots  & \vdots \\
                                                                                                         0 &1 & &\ldots & 0 & 0 \\
                                                                                                         1 & 0 & \ldots & \ldots &0 &0 \end{pmatrix}.$$
This means that the Kronecker index specifies the degrees of the polynomials on the diagonal of the autoregressive
polynomial $P(z)$; the polynomials on the secondary line have a degree of at most $\min(m_i+\1_{\{i>j\}},m_j)$. In particular, we can calculate
the degree $p=\max_{i=1,\ldots,d}m_i$ of the autoregressive polynomial.
 Moreover, the polynomials $P$ and $Q$ can be calculated explicitly from $A,B$ and $C$.
Important is that an MCARMA process in Echelon form  fulfills the smoothness and identifiability assumptions % desired in \autoref{as_D} automatically (cf.beginning of \cite[Chapter 4.1]{schlemmstelzer}),
$\ref{as_D4}$, $\ref{as_D5}$ and $\ref{as_D6}$.
% The rest of Assumption \ref{as_D} can be easily imposed by restricting $\Theta$ and the driving L\'{e}vy process in a suitable way.
A special subclass  of MCARMA processes in Echelon form are the one-dimensional CARMA processes, for which the degree $p$ of the autoregressive polynomial is fixed and the zeros of $P$ and $Q$ are distinct. This class
corresponds to the class of CARMA processes in Echelon form with Kronecker index $p$. For more details on MCARMA processes in Echelon form we refer to \cite[Section 4.1]{schlemmstelzer}.
\end{remark}

\subsection{Asymptotic normality} \label{Section:3.4}

The next proposition collects auxiliary results which are used in the proof of the asymptotic normality of the QMLE. They are highlighted here separately for easier reference, because they will appear again later in a different context.
\begin{proposition}\label{helpprop} \strut
\begin{itemize}
\item[(a)]
Assume that the space $\Theta$ with associated family of continuous-time state space models \linebreak $(A_{\vt}, B_{\vt}, C_{\vt}, L_{\vt})_{\vt \in \Theta}$ satisfies Assumptions $\ref{as_D1}$ to $\ref{as_D3}$ as well as $\ref{as_D9}$. Then, there exists a pseudo-true parameter $\vt^* \in \Theta$ as defined in Equation \eqref{defpseudotrue} and for every $n \in \N$, there exists
\begin{equation}\label{pseudotruewh}
\vt^\ast_n = \argmin_{\vt \in \Theta} \E \left[ \wh \LL ( \vt, Y^n) \right]
\end{equation}
 as well.
If $\Theta$ also satisfies the other parts of \autoref{as_D}, then  $\vt^\ast_n \to \vt^\ast$ as $n \to \infty$.
In particular, for  $n$ sufficiently large  $\vt^\ast_n$ is in the interior of $\Theta$ as well.
\item[(b)] Assume that the space $\Theta$ with associated family of continuous-time state space models \linebreak $(A_{\vt}, B_{\vt}, C_{\vt}, L_{\vt})_{\vt \in \Theta}$ satisfies Assumptions $\ref{as_D1}$ to $\ref{as_D5}$.
Then the strong law of large numbers
\begin{equation*}\label{SLLNlike}
 \wh \LL(\vt, Y^n)
 \to \mathscr{Q}(\vt) \quad % = d \log(2\pi) + \log(\det(V_\vt)) + \E \left[ \epsilon^T_{\vt,1} V^{-1}_\vt \epsilon_{\vt,1} \right]
 \p\text{-a.s.}
\end{equation*}
holds uniformly in $\vt$ as $n\to\infty$.
\item[(c)] Assume that the space $\Theta$ with associated family of continuous-time state space models \linebreak $(A_{\vt}, B_{\vt}, C_{\vt}, L_{\vt})_{\vt \in \Theta}$ satisfies \autoref{as_D}.
Then, as $n\to\infty$,
$$\sqrt{n}  \nabla_\vt \wh\LL( \vt^*, Y^n) \stackrel{\DD}{\to} \NN(0, \II(\vt^*)),$$
where $\II(\vt^*) = \lim_{n \to \infty} n \Var (\nabla_\vt \LL(\vt^*, Y^{n}))$.
\item[(d)] Assume that the space $\Theta$ with associated family of continuous-time state space models \linebreak $(A_{\vt}, B_{\vt}, C_{\vt}, L_{\vt})_{\vt \in \Theta}$ satisfies
Assumptions $\ref{as_D1}$ to $\ref{as_D5}$.
Then the convergence
\begin{equation*}\label{SLLNscnd}
\nabla^2_\vt \wh\LL( \vt, Y^n )
 {\to} \HH(\vt) \quad \p\text{-a.s.}
\end{equation*}
 holds uniformly in $\vartheta$ as $n\to\infty$, where $\HH(\vt):=\E \left[ \nabla^2_\vt l_{\vt, 1} \right]$.
\item[(e)] Assume that the space $\Theta$ with associated family of continuous-time state space models \linebreak $(A_{\vt}, B_{\vt}, C_{\vt}, L_{\vt})_{\vt \in \Theta}$ satisfies \autoref{as_D}. Then there exist $\epsilon, \alpha > 0$ such that for almost all $\omega$ and for every $n > n_1(\omega)$ and $\vt \in B_{\epsilon}(\vt^*) \cap \Theta$ we have
    $$\det \left(  \nabla^2_\vt \wh\LL( \vt, Y^n)(\omega)  \right) \geq \alpha.$$
\end{itemize}
\end{proposition}
\begin{proof} $\mbox{}$
(a) The existence statements follow directly from \cite[Proposition 3.1]{sinwhite}. The convergence $\vt^\ast_n \rightarrow \vt^\ast$ follows from \Cref{helemma}(d).\\
(b) This is exactly  \cite[Lemma 2.8]{schlemmstelzer} taking  \cite[Lemma 3.14]{schlemmstelzer} into account.\\
(c) Note that %by \cite[Proposition 3.1d)]{sinwhite}
 under \autoref{as_D} we have
%\begin{equation*}
$\nabla_\vt\mathcal{Q}(\vartheta)\big|_{\vt = \vt^*}=\nabla_\vt \E \left[ \LL(\vt, Y^n) \right] \big|_{\vt = \vt^*} = 0.$
%\end{equation*}
 Next, we use the dominated convergence theorem to interchange the expectation and derivation, giving
\begin{equation}\label{firstdevexp}
\E \left[ \nabla_\vt \LL(\vt, Y^n)  \right] \big|_{\vt = \vt^*} = 0.
\end{equation}
This rest of the proof can  be carried out as  \cite[Lemma 2.16]{schlemmstelzer}.\\
(d) The pointwise convergence can be proved as in \cite[Lemma 2.17]{schlemmstelzer}, respectively \cite[Lemma 2 and Lemma 3]{boubafrancq}
taking  \cite[Lemma 3.14]{schlemmstelzer} into account. The stronger statement of uniform convergence  can be shown
by using the compactness of the parameter space analogous to the proof of \cite[Lemma 2.16]{schlemmstelzer}, respectively
\cite[Theorem 16]{Ferguson}.
\\
(e) Assumption $\ref{as_D11}$ says that the Fisher information matrix $\E \left[ \nabla^2_\vt l_{\vt^*, 1} \right]$ is invertible
and hence, $\det(\E \left[ \nabla^2_\vt l_{\vt^*, 1} \right])>0$. Moreover, by Assumption \ref{as_D4} the map $\vt \mapsto \E \left[ \nabla^2_\vt l_{\vt, 1} \right]$
is continuous. Thus, there exist $\epsilon, \alpha >0$ such that $\inf_{\vt\in B_\epsilon(\vt^*)\cap\Theta}\det(\E \left[ \nabla^2_\vt l_{\vt, 1} \right])> \alpha$.
Since by (d) as $n\to\infty$,
$$\sup_{\vt\in B_\epsilon(\vt^*)\cap\Theta}\|\nabla^2_\vt \wh\LL( \vt, Y^n ) -\E \left[ \nabla^2_\vt l_{\vt, 1} \right]\|\to 0  \quad \p\text{-a.s.},$$
we finally get $\lim_{n \to \infty} \inf_{\vt\in B_\epsilon(\vt^*)\cap\Theta}\det(\nabla^2_\vt \wh \LL( \vt, Y^n ))> \alpha$ $\p$-a.s.
\end{proof}

We can now state the desired central limit theorem, which basically combines \cite[Proposition 4.1]{sinwhite} and \cite[Theorem 3.16]{schlemmstelzer}.
\begin{theorem}\label{idprop} \strut
Assume that the space $\Theta$ with associated family of continuous-time state space models $(A_{\vt}, B_{\vt}, C_{\vt}, L_{\vt})_{\vt \in \Theta}$ satisfies \autoref{as_D}. Then, as $n\to\infty$,
\begin{eqnarray*}
    \widehat{\vt}^n \to \vt^* \quad \Pas,
\end{eqnarray*}
and
$$\sqrt{n} \left( \widehat{\vt}^n - \vt^* \right) \stackrel{\DD}{\to} \NN(0, {\HH}^{-1}(\vt^*) \II(\vt^*) {\HH}^{-1}(\vt^*)),$$
where
\begin{equation}\label{vtcovmiss}
\II(\vt^*) = \lim_{n \to \infty} n \Var (\nabla_\vt \LL(\vt^*, Y^{n})) \quad \text{  and  } \quad\HH(\vt^*) = \lim_{n \to \infty}  \nabla^2_\vt \LL(\vt^*, Y^{n}).
\end{equation}
\end{theorem}
\begin{proof}
The proof can be carried out in the same way as
% the ones of
\cite[Theorem 3.16, Theorem 2.4 and Theorem 2.5]{schlemmstelzer}, respectively, replacing $\vt_0$ by $\vt^*$ wherever it appears. Note that we have the additional assumption $\ref{as_D10}$ concerning identifiable uniqueness, which ensures that the estimator converges to a unique limit, see also \cite[Theorem 3.4]{white1996}.
\end{proof}
\begin{remark} $\mbox{}$\\
(a) \, For the strong consistency part of the theorem, Assumption $\ref{as_D3}$ can be relaxed requiring only continuity instead of three times differentiability.\\
(b) \, In the case that we are in a correctly specified parameter space, this theorem corresponds exactly to \cite[Theorem 3.16]{schlemmstelzer}.
% , i. e. it is a direct extension of it to the more general case.
\end{remark}

\subsection{Law of the iterated logarithm} \label{Section:3.5}
This section is devoted to the development of various forms of the law of the iterated logarithm which we  need to study the consistency properties of the information criteria. In the following proposition we start by establishing a law of the iterated logarithm for linear combinations of partial derivatives of the quasi log-likelihood function.
\begin{proposition}\label{likelihoodlil}
Assume that the space $\Theta$ with associated family of continuous-time state space models $(A_{\vt}, B_{\vt}, C_{\vt}, L_{\vt})_{\vt \in \Theta}$
satisfies \autoref{as_D}.
Then, for every $x \in \R^{N(\Theta)}\setminus\{0_{N(\Theta)}\}$
% and every $i \in \{1, \ldots, N(\Theta)\}$
 it holds that
\begin{equation*}
\begin{split}\label{1devlimsup}
\limsup_{n \to \infty}  \frac{\sqrt{n}}{\sqrt{ \log(\log(n))}} x^T \nabla_\vt \wh\LL( \vt^*, Y^n)  &=  \sqrt{2\cdot x^T \II(\vt^\ast) x} \quad \Pas, \\
\liminf_{n \to \infty}  \frac{\sqrt{n}}{\sqrt{ \log(\log(n))}} x^T \nabla_\vt \wh\LL( \vt^*, Y^n)   &= -\sqrt{2\cdot x^T \II(\vt^\ast) x} \quad \Pas
\end{split}
\end{equation*}
\end{proposition}
\begin{proof}
Let $x \in \R^{N(\Theta)}\setminus\{0_{N(\Theta)}\}$.
First, it can be deduced that  $x^T \II(\vt^\ast) x$ is finite and positive from \cite[Lemma 2.16]{schlemmstelzer}.
Moreover, by \cite[Eq. (2.24)]{schlemmstelzer} the representation
\begin{equation}\label{partdevl}
\partial_i l_{\vt^*, k} = \tr \left( V^{-1}_{\vt^*} \left( I_{d \times d} - \epsilon_{\vt^*, k} \epsilon^T_{\vt^*, k} V^{-1}_{\vt^*} \right) \partial_i V_{\vt^*} \right) + 2 \partial_i \epsilon^T_{\vt^*, k} V^{-1}_{\vt^*} \epsilon_{\vt^*, k}
\end{equation}
holds. By \Cref{serieslemma} we know that both the pseudo-innovations and their partial derivatives can be expressed as moving averages of the true output process via
\begin{equation}\label{seriesrep4}
\epsilon_{\vt^*,k} = \sum_{\nu = 0}^{\infty} c_{\vt^*, \nu} Y ((k-\nu)h), \quad \quad
\partial_i \epsilon_{\vt^*,k} = \sum_{\nu = 0}^{\infty} c^{(i)}_{\vt^*, \nu} Y((k-\nu)h)
\end{equation}
and the inequalities $\sup_{\vt \in \Theta} \| c_{\vt, \nu} \| \leq C \rho^\nu$ and $\sup_{\vt \in \Theta} \| c^{(i)}_{\vt, \nu} \| \leq {C} {\rho}^\nu$ are satisfied for some \linebreak $C > 0$ and $\rho \in (0,1)$ for $i \in \{ 1, \ldots, N(\Theta) \}$. Thus, $x^T \nabla_\vt l_{\vt^*, k} = \sum_{i=1}^{N(\Theta)} x_i \partial_i l_{\vt^\ast, k}$ can be written as $f(Y(kh), Y((k-1)h), \ldots)$ for a suitable function $f$.\\
The aim is now to apply the law of the iterated logarithm  for dependent random variables
as it's given in \cite[Theorem 8]{oodaira}, for which we need to check the following three conditions:
\begin{itemize}
\item[(a)] $\E \left[x^T \nabla_\vt l_{\vt^*, k} \right] = 0$ and $\E  \left| x^T \nabla_\vt l_{\vt^*, k} \right|^{2+\delta_1}  < \infty$ for some $\delta_1 > 0$.
\item[(b)] $\E \left[ \left| x^T \nabla_\vt l_{\vt^*, k} - \E \left[ x^T \nabla_\vt l_{\vt^*, k} \mid \sigma\left( Y(( k-m)h), \ldots, Y(kh), \ldots, Y((k+m)h) \right) \right] \right|^2 \right]$ \linebreak
$= O(m^{-2-\delta_2})$ for some $\delta_2 > 0$ and $m \in \N$. %, where the $\sigma$-field in the conditional expectation is the one generated by the random variables in the parentheses.
\item[(c)] $\sum_{k=1}^{\infty} \alpha_{Y^{(h)}}(k)^{\frac{\delta_3}{2+\delta_3}} < \infty$ for some $0 < \delta_3 < \delta_1$, where $( \alpha_{Y^{(h)}}(k) )_{k \in \Z}$ denotes the strong mixing coefficients of the process $( Y(k h))_{k \in \Z}$.
\end{itemize}
 (a)\, We start with the first condition. For the first part it follows as in \eqref{firstdevexp} that $\E \left[ \partial_i l_{\vt^\ast, k} \right] = 0$ for every $i\in\{1,\ldots,N(\Theta)\}$, hence $\E \left[x^T \nabla_\vt l_{\vt^*, k} \right] = 0$. For the second part, for any $i\in\{1,\ldots,N(\Theta)\}$ we employ   $\eqref{partdevl}$ and the Cauchy-Schwarz inequality to obtain
\begin{eqnarray*}
\E  \left| \partial_i l_{\vt^*, k} \right|^{2+\delta_1} &\leq& C \E \left| \tr \left( V^{-1}_{\vt^*} \epsilon_{\vt^*, k} \epsilon^T_{\vt^*, k} V^{-1}_{\vt^*} \partial_i V_{\vt^*} \right) \right|^{2+\delta_1} + C \E \left| \partial_i \epsilon^T_{\vt^*, k} V^{-1}_{\vt} \epsilon_{\vt^*,k} \right|^{2+\delta_1} \\
&\leq & C \left( \E  \| \epsilon_{\vt^*, k} \|^{4+ 2 \delta_1}  + \left( \E \| \epsilon_{\vt^*, k} \|^{4+2 \delta_1} \E  \| \partial_i \epsilon_{\vt^*, k} \|^{4+2 \delta_1}  \right)^{\frac12} \right) ,
\end{eqnarray*}
where we have used the  the compactness of $\Theta$ in the last line. From Assumption $\ref{as_D9}$  we know that the driving L\'{e}vy process $L$
 of $Y$ has finite $(4+\delta)$th moment for some $\delta > 0$, which carries over to the $(4+\delta)$th  moment of $Y(kh)$, $k\in\Z$, and hence to $\epsilon_{\vt^*, k}$ and $\partial_i \epsilon_{\vt^*, k}$. With this, we obtain that
% From the MA representations in \eqref{seriesrep4}
%  and the exponential bounds given right after, we receive
% $$\max\left(\E \| \partial_i \epsilon_{\vt^*, k} \|^{4+2 \delta_1} ,\E  \| \epsilon_{\vt^*, k} \|^{4+ 2 \delta_1} \right) \leq \left( C \sum_{\nu=0}^{\infty} \rho^\nu \right)^{4 + 2 \delta_1} \E \left\| Y(h)  \right\|^{4+2 \delta_1}.$$
the right-hand side is finite if $\delta_1 < \frac{\delta}{2}$.
Since $i\in\{1,\ldots,N(\Theta)\}$ is arbitrary and $x^T \nabla_\vt l_{\vt^*, k}$ is a linear combination of those components,
 we get $\E  \left| x^T \nabla_\vt l_{\vt^*, k} \right|^{2+\delta_1}  < \infty$.\\
(b) \, For the second condition, we begin by decomposing the partial derivative as in the proof of \cite[Lemma 2.16]{schlemmstelzer}. For $m \in \N$ we write
$$\partial_i l_{\vt^*, k} = Y^{(i)}_{m,k} - \E \left[ Y^{(i)}_{m,k} \right] + Z^{(i)}_{m,k} - \E \left[ Z^{(i)}_{m,k} \right],$$
where
\begin{eqnarray*}
Y^{(i)}_{m,k} &=& \tr\left( V^{-1}_{\vt^*} \partial_i V_{\vt^*} \right) + \sum_{\nu, \nu'=0}^{m} \left( -\tr \left( V^{-1}_{\vt^*} c_{\vt^*, \nu} Y((k - \nu)h) Y^{T}_{\vt_0}(( k - \nu')h) c^{T}_{\vt^*, \nu'} V^{-1}_{\vt^*} \partial_i V_{\vt^*} \right) \right. \\
             &&\hspace*{4.3cm}+ \left. 2 Y^{T}_{\vt_0}(( k-\nu)h) c^{(i),T}_{\vt^*, \nu} V^{-1}_{\vt^*} c_{\vt^*, \nu'} Y(( k - \nu')h) \right), \\
Z^{(i)}_{m,k} &=& \partial_i l_{\vt^*, k} - Y^{(i)}_{m,k}.
\end{eqnarray*}
Hence, we obtain
\begin{align*}
&\E \left[ \left| x^T \nabla_\vt l_{\vt^*, k} - \E \left[\left. x^T \nabla_\vt l_{\vt^*, k} \right| \sigma\left( Y(( k-m)h), \ldots, Y(kh), \ldots, Y((k+m)h) \right) \right] \right|^2 \right] \notag \\
%&\quad=\,\, \E \left[ \left| \sum_{i=1}^{N(\Theta)} x_i Z^{(i)}_{m,k} - \E \left[\left. \sum_{i=1}^{N(\Theta)} x_i Z^{(i)}_{m,k} \right| \sigma\left( Y_{\vt_0}(( k-m)h), \ldots, Y_{\vt_0}(kh), \ldots, Y_{\vt_0}((k+m)h) \right)  \right] \right|^2 \right]\\
& \leq \E  \left[ \left| \sum_{i=1}^{N(\Theta)} x_i Z^{(i)}_{m,k} - \E \left[ \sum_{i=1}^{N(\Theta)} x_i Z^{(i)}_{m,k} \right] \right|^2 \right] % = \Var \left( \sum_{i=1}^{N(\Theta)} x_i Z^{(i)}_{m,k} \right) \notag \\
= \sum_{i=1}^{N(\Theta)} x^2_i \Var ( Z^{(i)}_{m,k} ) + 2  \sum_{\substack{i,j=1 \\ i \neq j}}^{N(\Theta)} x_i x_j \Cov( Z^{(i)}_{m,k}, Z^{(j)}_{m,k}).
\end{align*}
 From step $2$ of the proof of \cite[Lemma 2.16]{schlemmstelzer} we know that $\Cov(Z^{(i)}_{m,k}, Z^{(j)}_{m,k}) \leq {C} {\rho}^m$ for a positive constant ${C}$ and ${\rho} \in (0, 1)$, and every $i, j \in \{ 1, \ldots, N(\Theta)\}$. Thus, the second condition is satisfied as well.\\
 (c)\, Lastly, we turn to the third condition. By \cite[Proposition 3.34]{marquardtstelzer} the strong mixing coefficients $\alpha_{Y}(t)$ of $(Y(t))_{t \in \R}$ are $O(\e^{-at})$ for some $a > 0$, which carries over to those of the sampled process $(Y(kh))_{k \in \Z}$. Thus, we can choose $\delta_3 < \delta_1 < \frac{\delta}{2}$ to obtain $\sum_{k=1}^{\infty} \alpha_{Y^{(h)}}(k)^{\frac{\delta_3}{2+\delta_3}} < \infty$ as desired.

Then a consequence of (a)-(c) and \cite[Theorem 8]{oodaira} is the law of the iterated logarithm
\begin{eqnarray*}
\limsup_{n \to \infty} \frac{ \left| \sum_{k=1}^{n} (\sum_{i=1}^{N(\Theta)} x_i \partial_i l_{\vt^{*},k}) \right|}{\sqrt{ 2 n x^T \II(\vt^\ast) x \log(\log(n x^T \II(\vt^\ast) x))}}
= 1 \quad \p\text{-a.s.}
\end{eqnarray*}
Since $\log(\log(n x^T \II(\vt^\ast) x )) = O(\log(\log(n)))$ we can therefore  deduce  the statement  by symmetry (the driving L\'{e}vy process has
expectation $0_s$) for $\mathcal{L}$.
%\begin{equation}
%\begin{split}
%\limsup_{n \to \infty}  \frac{\sqrt{n}}{\sqrt{ \log(\log(n))}} x^T \nabla_\vt \LL( \vt^*, Y^n_{\vt_0})  &=  \sqrt{2\cdot x^T \II(\vt^\ast) x} \quad \Pas, \label{lilscalarprod} \\
%\liminf_{n \to \infty}  \frac{\sqrt{n}}{\sqrt{ \log(\log(n))}} x^T \nabla_\vt \LL( \vt^*, Y^n_{\vt_0})  &=-  \sqrt{2\cdot x^T \II(\vt^\ast) x} \quad \Pas
%\end{split}
%\end{equation}
Finally, by \Cref{helemma}(b) we can transfer the result to $\wh{\mathcal{L}}$ as well.
\end{proof}

The next theorem builds upon this to derive a multivariate version of the law of the iterated logarithm.
\begin{theorem}\label{LLmult}
Assume that the space $\Theta$ with associated family of continuous-time state space models $(A_{\vt}, B_{\vt}, C_{\vt}, L_{\vt})_{\vt \in \Theta}$
satisfies \autoref{as_D}. Moreover, let $\Xi \in \R^{N(\Theta) \times N(\Theta)}$ be an arbitrary matrix. Then it holds that
\begin{equation*}\label{LLmultlim}
\limsup_{n \to \infty} \frac{\sqrt{n}}{\sqrt{ \log(\log(n))}} \| \Xi \nabla_\vt \widehat{\LL}(\vt^\ast, Y^n) \|
=  \sqrt{2\cdot\lambda_{\text{max}}(\Xi \II(\vt^\ast) \Xi^T)} \quad \Pas
\end{equation*}
\end{theorem}
\begin{proof}
First, since $\II(\vt^\ast) = \lim_{n \to \infty} n \Var \left( \nabla_\vt \LL( \vt^\ast, Y^n) \right)$ (cf. \Cref{helpprop}(c)), it holds that
$$\lim_{n \to \infty} n \Var( \Xi \nabla_\vt \LL( \vt^*, Y^n) ) = \Xi \II(\vt^\ast) \Xi^T. $$
An application of \Cref{likelihoodlil} gives
$$\limsup_{n \to \infty}  \frac{\sqrt{n}}{\sqrt{ \log(\log(n))}} x^T \Xi \nabla_\vt \wh \LL( \vt^*, Y^n)  =   \sqrt{2\cdot x^T \Xi \II(\vt^\ast) \Xi^T x} \quad \Pas$$
for every $x \in \R^{N(\Theta)}\setminus\{0_{N(\Theta)}\}$.
%Using the fact that $\R^{N(\Theta)}$ is its own dual space and viewing the mapping $x \mapsto x^T \Xi \nabla_\vt \wh \LL( \vt^*, Y^n_{\vt_0})$ as the application of the linear functional $x$ to $\Xi \nabla_\vt \wh \LL( \vt^*, Y^n_{\vt_0})$, this means that a law of the iterated logarithm holds for every univariate process of the form $x(\Xi \nabla_\vt \wh \LL( \vt^*, Y^n_{\vt_0}))$.
Just as in the proof of \cite[Lemma 2]{finkelstein}, we can conclude from this that $\Pas$
\begin{eqnarray*}
{\limsup_{n \to \infty} \frac{\sqrt{n}}{\sqrt{ \log(\log(n))}} \| \Xi \nabla_\vt \wh {\LL}(\vt^\ast, Y^n) \|}
 &=&\limsup_{n \to \infty} \frac{\sqrt{n}}{\sqrt{ \log(\log(n))}} \sup_{\|x\|=1}\left|x^T \Xi \nabla_\vt \wh {\LL}(\vt^\ast, Y^n)\right|\\
 &=&\sup_{\| x \| = 1 } \sqrt{2\cdot x^T \Xi \II(\vt^\ast) \Xi^T x}\\
&=& \sqrt{2\cdot\lambda_{\text{max}}(\Xi \II(\vt^\ast) \Xi^T)}.
\end{eqnarray*}
%where \cite[Eq. (22)]{zhulenev} can be taken into account to deal with the fact that the covariance matrix of $\Xi \nabla_\vt {\LL}(\vt^\ast, Y^n_{\vt_0})$ is not necessarily the identity (since we are in a finite-dimensional Hilbert space, see also \cite[pp. 222 and 232]{ledouxtalagrand} for the identification of the limit). \marginpar{check}
\end{proof}
Having this theorem allows us to derive a variant of the law of the iterated logarithm for the function $\wh \LL$.
\begin{theorem}\label{LLasymp}
Assume that the space $\Theta$ with associated family of continuous-time state space models $(A_{\vt}, B_{\vt}, C_{\vt}, L_{\vt})_{\vt \in \Theta}$ satisfies \autoref{as_D}.
Then
\begin{eqnarray*}
\limsup_{n \to \infty} \frac{n}{\log(\log(n))} \left(  \wh \LL( \vt^*, Y^n) - \wh\LL( \widehat{\vt}^n, Y^n) \right) = \lambda_{\text{max}}(\HH(\vt^\ast)^{-\frac12} \II(\vt^\ast) \HH(\vt^\ast)^{-\frac12}) \quad \Pas \label{asymptoticvt2}
\end{eqnarray*}
\end{theorem}

\begin{proof}
A first-order Taylor expansion of $ \nabla_\vt \wh\LL( \widehat{\vt}^n, Y^n) $ around $\vt^\ast$ gives
\begin{equation*}
0 =  \nabla_\vt \wh\LL( \widehat{\vt}^n, Y^n)  =  \nabla_\vt \wh\LL( \vt^*, Y^n)  +  \nabla^2_\vt \wh\LL( \overline{\vt}^n, Y^n)  (\widehat{\vt}^n - \vt^*),
\end{equation*}
for some $\overline{\vt}^n$ with $\| \overline{\vt}^n - \vt^\ast \| \leq \| \wh \vt^n - \vt^\ast\|$.
Since by  \Cref{idprop} we know that $\widehat{\vt}^n\to \vt^*$ $\Pas$, $\overline{\vt}^n\to \vt^*$ $\Pas$ as well. A conclusion of \Cref{helpprop}(e) is that
 $\lim_{n\to\infty}\det(\nabla^2_\vt \wh\LL( \overline{\vt}^n, Y^n))>0$  $\Pas$, so that
\begin{equation}\label{invL2}
\widehat{\vt}^n - \vt^* = - \left(\nabla^2_\vt \wh\LL( \overline{\vt}^n, Y^n)   \right)^{-1} \nabla_\vt \wh\LL( \vt^*, Y^n) \quad \Pas
\end{equation}
is well-defined.
Now we employ a Taylor expansion again, albeit this time we expand $\wh\LL(\vt^*, Y^n)$ around $\widehat{\vt}^n$ and use a second-order expansion. This gives us
\begin{eqnarray*}
 \wh\LL (\vt^*, Y^n)
  = \wh\LL (\widehat{\vt}^n, Y^n) + \frac12 (\widehat{\vt}^n - \vt^*)^T \nabla^2_\vt \wh\LL( \check{\vt}^n, Y^n)  (\widehat{\vt}^n - \vt^*),
\end{eqnarray*}
 for some $\check{\vt}^n$ with $\| \check{\vt}^n - \wh{\vt}^n \| \leq \| \wh \vt^n - \vt^\ast\|$, where we have used $\nabla_\vt \wh\LL( \widehat{\vt}^n, Y^n)  = 0$. As above we have $\check{\vt}^n \to \vt^\ast$ $\Pas$ Rearranging the terms, we arrive at
 \begin{align}
\wh\LL( \vt^*, Y^n) - \wh\LL( \widehat{\vt}^n, Y^n) &= \frac12 \| \nabla^2_\vt \wh\LL( \check{\vt}^n, Y^n)^{\frac12} (\widehat{\vt}^n - \vt^*) \|^2  \notag \\
                                                                    &= \frac12 \| \nabla^2_\vt \wh\LL( \check{\vt}^n, Y^n)^{\frac12} (\nabla^2_\vt \wh\LL( \overline{\vt}^n, Y^n))^{-1} \nabla_\vt \wh\LL( \vt^*, Y^n) \|^2. \label{widehatLL}
 \end{align}
An application of \Cref{LLmult} with $\Xi = \HH(\vt^\ast)^{-\frac12}$ (which is symmetric)
yields
\begin{eqnarray}\label{widehatLL2}
{\limsup_{n \to \infty} \frac{\sqrt{n}}{\sqrt{ \log(\log(n))}} \| \HH(\vt^\ast)^{-\frac12} \nabla_\vt \widehat{\LL}(\vt^\ast, Y^n) \|} \nonumber
=  \sqrt{2\cdot\lambda_{\text{max}}(\HH(\vt^\ast)^{-\frac12} \II(\vt^\ast) \HH(\vt^\ast)^{-\frac12})} \,\,\,\, \Pas
\end{eqnarray}
With $  \nabla^2_\vt \wh\LL( \check{\vt}^n, Y^n)^{\frac12} \nabla^2_\vt \wh\LL( \ov{\vt}^n, Y^n)^{-1}\to  \HH(\vt^\ast)^{-\frac12}$ $\p$-a.s. (cf. \Cref{helpprop}(d)) and \eqref{widehatLL}
we can derive the statement.
\end{proof}

\begin{remark}\label{remarkwhite}
This result is an analog to \cite[Proposition 5.1]{sinwhite}
 which investigates consistency of information criteria  under some different model assumptions.
However, it is stronger than the one in the cited article, since we are able to specify the limit superior
exactly while in \cite{sinwhite} it is only shown that convergence occurs.
\end{remark}

\section{Likelihood-based information criteria}\label{sec:IC} \label{Section:4}

In this main  section we  derive properties for likelihood-based information criteria of the
following form.

\begin{definition}\label{ICLdef}
Assume that the space $\Theta$ with associated family of continuous-time state space models $(A_{\vt}, B_{\vt}, C_{\vt}, L_{\vt})_{\vt \in \Theta}$ satisfies \autoref{as_D}.
Furthermore, let $\widehat{\vt}^n_{}$ be the QMLE based on $Y^n$ in $\Theta_{}$ as defined in \eqref{GQMLE} and let $C(n)$ be a positive, nondecreasing function of $n$ with
$$\lim_{n \to \infty} \frac{C(n)}{n} = 0.$$
Then a likelihood-based \textrm{information criterion} has the form
\begin{equation}\label{ICL}
\text{IC}_n(\Theta_{}) :=\wh \LL ( \widehat{\vt}^n, Y^n) + N(\Theta) \frac{C(n)}{n}.
\end{equation}
\end{definition}
These information criteria have the property that $\text{IC}_n(\Theta_{})\stackrel{\p}{\to}\mathcal{Q}(\vt)$. Since
$\mathcal{Q}$ attains its minimum at
$\vt_0$ for which $\text{MCARMA}(A_{\vt_0},B_{\vt_0},C_{\vt_0},L_{\vt_0})=Y$ (cf. \Cref{loglikelemma}) we choose the parameter space for which the information criterion is minimal. The condition
$C(n)/n \to 0$ guarantees that underfitting is not possible, i. e. there is no positive probability
of choosing a parameter space which cannot generate the process underlying the data. However,
$C(n)/n \to 0$ is not sufficient to exclude overfitting, i.e.  a positive probability to choose
a space with more parameters than necessary. In the following we will give necessary and sufficient
conditions to exclude this case. To this end we  need some notation.

\begin{definition} \label{def:4.2}
Let $\Theta$ and $\Theta_0$ be parameter spaces with associated families of continuous-time state space models
$(A_{\vt}, B_{\vt}, C_{\vt}, L_{\vt})_{\vt \in \Theta_0}$ and $(A_{\vt}, B_{\vt}, C_{\vt}, L_{\vt})_{\vt \in \Theta}$, respectively,
satisfying \autoref{as_D}.
 Assume that there is a $\vt_0 \in \Theta_0$ with $\text{MCARMA}(A_{\vt_0},B_{\vt_0},C_{\vt_0},L_{\vt_0})=Y$.
% \begin{itemize}
% \item[(a)] We say that $\Theta_0$ is \textbf{disjoint from $\Theta$} if for every $\vt \in \Theta$ it holds that $\text{MCARMA}(A_{\vt},B_{\vt},C_{\vt},L_{\vt}) \neq Y$.
% \item[(b)]
We say that $\Theta_0$ is nested in $\Theta$ if $N(\Theta_0)<N(\Theta)$ and there exist a matrix $F \in \R^{N(\Theta) \times N(\Theta_0)}$  with $F^TF = I_{N(\Theta_0) \times N(\Theta_0)}$ as well as a $c \in \R^{N(\Theta)}$ such that %for the linear map $\vt      \mapsto F \vt + c$
  $$(A_{\vt}, B_{\vt}, C_{\vt}, L_{\vt})_{\vt \in \Theta_0}=(A_{F \vt + c}, B_{F \vt + c}, C_{F \vt + c}, L_{F \vt + c})_{\vt \in \Theta_0}.$$
% \end{itemize}
\end{definition}

The interpretation of nested is that all processes generated by a parameter in $\Theta_0$ can also be generated by a parameter in $\Theta$. However, there
are also processes which can be generated by a parameter in $\Theta$, but not by a parameter in $\Theta_0$. In this sense $\Theta_0$ is
contained in $\Theta$. The condition $F^TF = I_{N(\Theta_0) \times N(\Theta_0)}$ guarantees that we have a bijective map from $\Theta_0\to F\Theta_0+c\subset\Theta$.

For MCARMA processes parametrized in Echelon form, a parameter space $\Theta$ that satisfies \linebreak \autoref{as_D} contains
only processes that  have the same Kronecker index $m=(m_1,\ldots,m_d)$
and hence, fixed degree $p=\max_{i=1, \ldots, d} m_i$ of the AR polynomial. However, for the MA polynomial we only know that the degree is less than or equal to $p-1$.
In this context $\Theta_0$ could be a parameter space generating processes with Kronecker index $m_0$ and MA degree not exceeding $q_0$, where
$\Theta$ generates processes with Kronecker index $m_0$ and MA degree not exceeding $q$, $q_0<q\leq p_0-1$. Then $\Theta_0$ is nested in $\Theta$.
In this way our information criteria can be used to estimate the Kronecker index, the degree of the AR polynomial and the degree of the MA polynomial.
%(Note that we do not consider a space $\Theta$ generating processes with Kronecker index $m_0$ and MA degree $q < q_0$ as nested in $\Theta_0$ -- for every $\vt$ in such a space it holds that
%$\text{MCARMA}(A_{\vt},B_{\vt},C_{\vt},L_{\vt})\not=Y$.)

In the following we investigate only parameter spaces with associated family of continuous-time state space models
$(A_{\vt}, B_{\vt}, C_{\vt}, L_{\vt})$ in Echelon form.
Let the Kronecker index, the degree of the AR polynomial and the degree of the MA polynomial, respectively, belonging to $Y$ be denoted by
$m_0$, $p_0$ and $q_0$, respectively. Then $\Theta^E_0$ denotes the parameter space generating all MCARMA processes
with Kronecker index $m_0$. The degree of the AR polynomial of those processes is then $p_0$, the degree of the
 MA polynomial is between $0$ and $p_0-1$. The space $\Theta^E_0$ is the biggest parameter space generating MCARMA processes in Echelon form,
 satisfying \autoref{as_D} and containing a parameter $\vartheta_0^E$
with $\text{MCARMA}(A_{\vt_0^E},B_{\vt_0^E},C_{\vt_0^E},L_{\vt_0^E})=Y$. Note that $\vt^E_0$ is then the pseudo-true parameter in $\Theta^E_0$.%, in which $\Theta_0$ is nested.
 %Contrasting this, we denote by $\Theta_0$ the smallest space containing a parameter that generates the observed process, i. e. the processes in $\Theta_0$ have Kronecker index $m_0$, AR order $p_0$ and MA order not greater than $q_0$.

Next, we define under which circumstances  $IC_n$ is consistent; we distinguish two different types of consistency.
\begin{definition}\label{defcons} $\mbox{}$
%Let $\vartheta_0$ be given.
\begin{itemize}
\item[(a)] The information criterion $IC_n$ is called strongly consistent if for any parameter spaces $\Theta_0$ and $ \Theta$
with associated families of continuous-time state space models
$(A_{\vt}, B_{\vt}, C_{\vt}, L_{\vt})_{\vt \in \Theta_0}$ and $(A_{\vt}, B_{\vt}, C_{\vt}, L_{\vt})_{\vt \in \Theta}$, respectively,
satisfying \autoref{as_D} and with a
  $\vt_0\in\Theta_0$ such that \linebreak $\text{MCARMA}(A_{\vt_0},B_{\vt_0},C_{\vt_0},L_{\vt_0})=Y$, and either $\text{MCARMA}(A_{\vt},B_{\vt},C_{\vt},L_{\vt})\not=Y$  for every $\vt \in \Theta$ or $\Theta_0$ being nested in $\Theta$ we have
\begin{equation*}\label{defstrongcons}
\P\left( \limsup_{n\to\infty}\left(\text{IC}_n(\Theta_{0})- \text{IC}_n(\Theta)\right)<0 \right) = 1.
\end{equation*}
\item[(b)] The information criterion $IC_n$ is called weakly consistent if for any parameter spaces $\Theta_0$ and $ \Theta$
with associated families of continuous-time state space models
$(A_{\vt}, B_{\vt}, C_{\vt}, L_{\vt})_{\vt \in \Theta_0}$ and $(A_{\vt}, B_{\vt}, C_{\vt}, L_{\vt})_{\vt \in \Theta}$, respectively,
satisfying \autoref{as_D} and 
with a $\vt_0\in\Theta_0$ such that \linebreak $\text{MCARMA}(A_{\vt_0},B_{\vt_0},C_{\vt_0},L_{\vt_0})=Y$, and either $\text{MCARMA}(A_{\vt},B_{\vt},C_{\vt},L_{\vt})\not=Y$  for every $\vt \in \Theta$ or $\Theta_0$ being nested in $\Theta$ we have
\begin{equation*}\label{defweakcons}
\lim_{n \to \infty} \P\left( \text{IC}_n(\Theta_{0}) - \text{IC}_n(\Theta)<0 \right) = 1.
\end{equation*}
\end{itemize}
\end{definition}
If the information criterion is strongly consistent, then the chosen parameter space converges almost surely to the true parameter space. For a weakly consistent
information criterion we only  have convergence in probability.
Moreover, if we compare two parameter spaces  both containing a parameter that generates the true output process, then we choose
the parameter space with less parameters asymptotically almost surely  in the strongly consistent
case, whereas in the weakly consistent case we have convergence in probability.
This especially means overfitting is asymptotically excluded.

With these notions  we  characterize consistency of $IC_n$ for MCARMA processes in terms  of the penalty term $C(n)$.
\begin{theorem}\label{strongcons} $\mbox{}$
\begin{itemize}
\item[(a)]
The criterion $IC_n$ is strongly consistent if
$$\limsup_{n \to \infty} \frac{C(n)}{\log(\log(n))} > \lambda_{\text{max}}(\HH(\vt^E_0)^{-\frac12} \II(\vt^E_0) \HH(\vt^E_0)^{-\frac12}).$$
The information criterion is {not} strongly consistent if
$\limsup_{n \to \infty} C(n)/\log(\log(n))=0.$
\item[(b)]  The criterion ${IC_n}$ is weakly consistent if
$\limsup_{n \to \infty} C(n) = \infty.$
 If
$\limsup_{n \to \infty} C(n) < \infty$
then ${IC_n}$ is neither weakly nor strongly consistent.
\item[(c)] Let $\Theta$ and $\Theta_0$ be parameter spaces with associated families of continuous-time state space models
$(A_{\vt}, B_{\vt}, C_{\vt}, L_{\vt})_{\vt \in \Theta_0}$ and $(A_{\vt}, B_{\vt}, C_{\vt}, L_{\vt})_{\vt \in \Theta}$, respectively,
satisfying \autoref{as_D}.
 Assume that there is a $\vt_0 \in \Theta_0$ with $\text{MCARMA}(A_{\vt_0},B_{\vt_0},C_{\vt_0},L_{\vt_0})=Y$ and $\Theta_0$ is nested in $\Theta$ with map $F$.
Moreover, suppose   $\limsup_{n\to\infty}C(n)=C<\infty$.
Define
\begin{align*}
\MM_{F}(\vt^\ast):= -\HH^{-1}(\vt^*) +  F (F^T \HH(\vt^\ast) F)^{-1} F^T.
\end{align*}
 Then
\begin{equation*}\label{thmcons1}
\lim_{n \to \infty} \P ( \text{IC}_n(\Theta_{0}) - \text{IC}_n(\Theta_{}) > 0 ) = \P \left( \sum_{i=1}^{N(\Theta)-N(\Theta_0)} \lambda_i \chi^2_i > 2 [ N (\Theta) - N(\Theta_0)]C \right)>0,
\end{equation*}
where $(\chi^2_i)$ is a sequence of independent $\chi^2$ random variables with one degree of freedom and
the $\lambda_i$ are the $N(\Theta)-N(\Theta_0)$ strictly positive eigenvalues of
$$\HH(\vt^*)^{\frac12} \MM_F(\vt^*) \II(\vt^*) \MM_F(\vt^*) \HH(\vt^*)^{\frac12}.$$
\end{itemize}
\end{theorem}
\begin{proof}
For the whole proof, we denote by $\vt_0$ the parameter in $\Theta_{0}$ with \linebreak $\text{MCARMA}(A_{\vt_0},B_{\vt_0},C_{\vt_0},L_{\vt_0})=Y$
% , which is nested in $\Theta_0^E$,
and by $\vt^*$ the pseudo-true parameter
 in $\Theta_{}$. Moreover, let $\widehat{\vt}^n_0$ denote the QMLE based on $Y^{n}$ in $\Theta_{0}$, $\widehat{\vt}^{n}_{}$ the QMLE based on $Y^n$ in $\Theta_{}$ and $\widehat{\vt}^E_0$ the
 QMLE based on $Y^n$ in $\Theta^E_0.$ The corresponding quasi log-likelihood functions are denoted by $\wh \LL_0$, $\wh \LL$ and $\wh {\LL}_E$, respectively.\\
\textsl{(a)} \,
We distinguish two different cases.\\
Case 1: $\text{MCARMA}(A_{\vt},B_{\vt},C_{\vt},L_{\vt})\not=Y$  for every $\vt \in \Theta$. Then
\begin{eqnarray}
\text{IC}_n(\Theta_{0}) - \text{IC}_n(\Theta_{}) &= \wh \LL_0 \left( \widehat{\vt}^n_{0}, Y^n \right) - \wh\LL \left( \widehat{\vt}^n, Y^n \right)
+ [N(\Theta_0) - N(\Theta)] \frac{C(n)}{n}. \label{strongcons1}
\end{eqnarray}
On the one hand, by \autoref{LLasymp} we have that
\begin{eqnarray*}
    \wh\LL \left( \widehat{\vt}^n, Y^n \right) &=& \wh \LL \left( \vt^*, Y^n \right) +O_{\text{a.s.}}\left( \frac{\log(\log(n))}{n} \right),\\
    \wh\LL_0 \left( \widehat{\vt}^n_{0}, Y^n \right) &=& \wh \LL _0\left( \vt_0, Y^n \right) + O_{\text{a.s.}} \left( \frac{\log(\log(n))}{n} \right),
\end{eqnarray*}
and on the other hand, by \Cref{helpprop}(b) % and \Cref{helemma}(a)
\begin{eqnarray*}
\wh  \LL \left( \vt^{*}, Y^n \right) =  \mathscr{Q}(\vt^\ast) + o_{\text{a.s.}}(1) \quad \mbox{and} \quad
 \wh  \LL_0 \left( \vt_0, Y^n \right)=  \mathscr{Q}(\vt_0) + o_{\text{a.s.}}(1).
\end{eqnarray*}
Finally, in this case the inequality from eq. \eqref{loglikelemmaeq2} is strict, so that for some $\delta>0$
\begin{align*}
\text{IC}_n(\Theta_{0}) - \text{IC}_n(\Theta_{}) &= \mathscr{Q}(\vt_0) - \mathscr{Q}(\vt^\ast)
                                         + \widehat{r}(n) + [N(\Theta_0) - N(\Theta)] \frac{C(n)}{n} \\
                                         &< - \delta + \widehat{r}(n) + [N(\Theta_0) - N(\Theta)] \frac{C(n)}{n},
\end{align*}
where $\widehat{r}(n)$ is $o_{\text{a.s.}}(1)$.
By assumption it holds that $ C(n)/n\to 0$ as $n \to \infty$, so that we get
\beao
    \p\left(\limsup_{n\to\infty}\left(\text{IC}_n(\Theta_{0}) - \text{IC}_n(\Theta_{})\right)<-\delta\right)=1.
\eeao
Case 2: $\Theta_0$ is nested in $\Theta$ with map $F$.
Note that $\Theta_0$ is also nested in $\Theta^E_0$ by definition, which then in turn means that $\Theta$ is nested in $\Theta^E_0$, implying
\begin{eqnarray} \label{4.2}
    \wh\LL (\wh\vt^n, Y^n) = \min_{\vt \in \Theta} \wh \LL(\vt, Y^n) \geq \min_{\vt \in \Theta^E_0} \wh \LL_E(\vt, Y^n) = \wh\LL_E (\wh\vt^n_E, Y^n) .
\end{eqnarray}
Moreover, $\wh\epsilon_{\vartheta_0,k}=\wh\epsilon_{\vartheta^*,k} = \wh \epsilon_{\vartheta^E_0,k}$ and hence,
\begin{equation}\label{proofcons1}
    \wh \LL_0 \left({\vt}_0, Y^n \right) = \wh\LL \left( \vt^*, Y^n \right) = \wh\LL_E \left( \vt^E_0, Y^n \right).
\end{equation}
With this and \eqref{4.2} we receive
\begin{eqnarray*}
\wh\LL_0 (\wh\vt^n_0, Y^n) - \wh\LL (\wh\vt^n, Y^n)% &\leq& \wh\LL_0 ( \vt_0, Y^n_{\vt_0}) - \wh\LL (\wh\vt^n, Y^n_{\vt_0}) \\
%&= &\wh\LL_E \left( \vt^E, Y^n_{\vt_0} \right) - \wh\LL (\wh\vt^n, Y^n_{\vt_0})\\
 \leq \wh\LL_E \left( \vt^E_0, Y^n \right) - \wh\LL_E (\widehat{\vt}^E_0, Y^n).
\end{eqnarray*}
Now, \Cref{LLasymp} tells us that
$$\limsup_{n \to \infty} \frac{n}{\log(\log(n))} \left(  \wh\LL_E \left( \vt^E_0, Y^n \right) - \wh\LL_E (\widehat{\vt}^E_0, Y^n)  \right) = \lambda_{\text{max}}(\HH(\vt^E_0)^{-\frac12} \II(\vt^E_0) \HH(\vt^E_0)^{-\frac12}) \quad \Pas$$
Turning to the information criterion, this gives
\begin{eqnarray*}
\lefteqn{\limsup_{n \to \infty}\frac{n}{\log(\log(n))}\left(\text{IC}_n(\Theta_{0}) - \text{IC}_n(\Theta_{})\right)} \notag \\
&\leq & \limsup_{n \to \infty} \frac{n}{\log(\log(n))} \left(\wh\LL_E \left( \vt^E_0, Y^n \right) - \wh\LL_E (\widehat{\vt}^E_0, Y^n)+
[N(\Theta_0) - N(\Theta)] \frac{C(n)}{\log(\log(n))} \right) \notag \\
&\leq& \lambda_{\text{max}}(\HH(\vt^E_0)^{-\frac12} \II(\vt^E_0) \HH(\vt^E_0)^{-\frac12})  - \limsup_{n \to \infty}\frac{C(n)}{\log(\log(n))} \quad \Pas, \label{lasteqcons}
\end{eqnarray*}
since $N(\Theta_0) - N(\Theta) \leq -1$. Hence, if $\limsup_{n \to \infty} \frac{ C(n)}{\log(\log(n))} > \lambda_{\text{max}}(\HH(\vt^E_0)^{-\frac12} \II(\vt^E_0) \HH(\vt^E_0)^{-\frac12})$, we  obtain
$$\p\left(\limsup_{n \to \infty} \frac{n}{\log(\log(n))} \left( \text{IC}_n(\Theta_{0}) - \text{IC}_n(\Theta_{}) \right) < 0\right)=1.$$
%and in particular, \marginpar{check}
%$\P( \text{IC}_n(\Theta_{0}) - \text{IC}_n(\Theta_{}) < 0) = 1$ for almost all $n$. \\
Finally, if $\limsup_{n \to \infty} C(n)/\log(\log(n))= 0$, then from $\wh \LL_0 \left(\wh{\vt}_0^n, Y^n \right)-\wh\LL \left( \wh \vt^n, Y^n \right)\geq 0$ it clearly follows that
$$\p\left(\limsup_{n \to \infty} \frac{n}{\log(\log(n))} \left( \text{IC}_n(\Theta_{0}) - \text{IC}_n(\Theta_{}) \right) > 0 \right)=1,$$
so that strong consistency cannot hold.\\
\textsl{(b)} \, Again we distinguish the two cases from part (a). Case 1 is dealt with analogously as in (a), so that we only need to give detailed arguments for case $2$. Suppose therefore that $\Theta_0$ is nested in $\Theta$. Define the map $f: \Theta_0 \to \Theta$ by $f(\vt) = F\vt +c$, where $F$ and $c$ are as in the definition of nested spaces. Then,
a Taylor expansion of $\wh\LL \left( f(\wh{\vt}^{n}_0), Y^{n} \right)$ around $\widehat{\vt}^{n}_{}$ results in
\begin{align}
\wh\LL_0 \left( \vt, Y^{n} \right)=\wh\LL \left( f(\wh{\vt}^{n}_0), Y^{n} \right) = \wh\LL \left( \widehat{\vt}^{n}_{}, Y^{n} \right) + \frac12 \left( \widehat{\vt}^{n}_{} - f(\wh{\vt}^{n}_0)\right)^T \nabla^2_\vt \wh\LL \left( \overline{\vt}^n, Y^{n} \right) \left( \widehat{\vt}^{n}_{} - f(\wh{\vt}^{n}_0) \right) \label{proofcons2}
\end{align}
with $\overline{\vt}^n$ such that $\| \overline{\vt}^n - \wh \vt^n \| \leq \| f( \wh \vt^n_0) - \wh \vt^n \|$.
Plugging  \eqref{proofcons2} into \eqref{strongcons1} gives
\begin{align}
\text{IC}_n(\Theta_{0}) - \text{IC}_n(\Theta_{}) = \frac12 \left( \widehat{\vt}^{n}_{} - f(\wh{\vt}^{n}_0) \right)^T \nabla^2_\vt \wh\LL \left( \overline{\vt}^n, Y^{n} \right) \left( \widehat{\vt}^{n}_{} - f(\wh{\vt}^{n}_0) \right) %\notag \\
+ [N(\Theta_0) - N(\Theta)] \frac{C(n)}{n}. \label{proofcons3}
\end{align}
In order to be able to show weak consistency, we will study the behavior of the random variable $\widehat{\vt}^{n}_{} - f(\wh{\vt}^{n}_0)$.
Note that $\wh\LL_0 \left( \vt, Y^{n} \right)=\wh\LL \left( f(\vt), Y^{n} \right)$ for $\vt\in\Theta_0$,
so that by the chain rule
$$\nabla_\vt \wh\LL_0( \vt_0, Y^n) = F^T\nabla_\vt \wh\LL( f(\vt_0), Y^n) = F^T \nabla_\vt \wh\LL( \vt^\ast, Y^n).$$
Moreover, % by a Taylor expansion we obtain
\begin{eqnarray*}
    f(\wh{\vt}^{n}_0)-\vartheta^*=f(\wh{\vt}^{n}_0)-f(\vt_0) = F (\widehat{\vt}^{n}_0 - \vt_0).
\end{eqnarray*}
%where $\overline{\vt}^n$ is such that $\| \overline{\vt}^n - \widehat{\vt}^{n}_{0} \| \leq \| \widehat{\vt}^{n}_{0} - \vt_0 \|$.
As in \eqref{invL2}, we also have
\begin{align*}
\widehat{\vt}^n - \vt^* = - \left(\nabla^2_\vt \wh\LL(  \check{\vt}^n, Y^n)   \right)^{-1} \nabla_\vt \wh\LL( \vt^*, Y^n), \\
\widehat{\vt}^{n}_0 - \vt_0 = - \left(\nabla^2_\vt \wh\LL_0( \tilde{\vt}^n, Y^n)   \right)^{-1} \nabla_\vt \wh\LL_0( \vt_0, Y^n),
\end{align*}
where $\check{\vt}^n$ is such that $\| \check{\vt}^n - \vt^\ast \| \leq \| \wh \vt^n - \vt^\ast\|$ and $\tilde{\vt}^n$ is such that $ \| \tilde{\vt}^n - \vt_0 \| \leq \| \widehat{\vt}^{n}_0 - \vt_0 \|$. In particular, $\check{\vt}^n\to \vt^\ast$ and $\tilde{\vt}^n\to\vt_0$ $\Pas$ as $n\to\infty$.
To summarize,
\begin{eqnarray*}
     \widehat{\vt}^{n}_{} - f(\wh{\vt}^{n}_0)
      &=&\widehat{\vt}^{n}_{}-\vt^*-F(\wh{\vt}^{n}_0-{\vt}_0)\\
    %    &=&- \left( \nabla^2_\vt \wh\LL( \check{\vt}^n, Y^n_{\vt_0})   \right)^{-1}  \nabla_\vt \wh\LL( \vt^*, Y^n_{\vt_0})- F(\widehat{\vt}^{n}_0 - \vt_0)\\
    %    &=&- \left( \nabla^2_\vt \wh\LL( \check{\vt}^n, Y^n_{\vt_0})   \right)^{-1}  \nabla_\vt \wh\LL( \vt^*, Y^n_{\vt_0})+F\left(  \nabla^2_\vt \wh\LL_0( \tilde{\vt}^n, Y^n_{\vt_0})   \right)^{-1}  \nabla_\vt \wh\LL_0( \vt_0, Y^n_{\vt_0})\\
 %       &=&- \left( \nabla^2_\vt \wh\LL( \check{\vt}^n, Y^n_{\vt_0})   \right)^{-1}  \nabla_\vt \wh\LL( \vt^*, Y^n_{\vt_0})+F\left(  \nabla^2_\vt \wh\LL_0( \tilde{\vt}^n, Y^n_{\vt_0})   \right)^{-1} F^T\nabla_\vt \wh\LL( f(\vt_0), Y^n_{\vt_0})\\
        &=&\left[- \left( \nabla^2_\vt \wh\LL( \check{\vt}^n, Y^n)\right)^{-1}+ F \left(  \nabla^2_\vt \wh\LL_0( \tilde{\vt}^n, Y^n)   \right)^{-1} F^T \right]  \nabla_\vt \wh\LL( \vt^*, Y^n).
\end{eqnarray*}
An application of \Cref{helpprop}(c) and (d) results in
\begin{align*}
    \sqrt{n}(\widehat{\vt}^{n}_{} - f(\wh{\vt}^{n}_0)) \stackrel{\mathcal{D}}{\to} \left[-\HH(\vt^*)^{-1}+ F \HH(\vt_0)^{-1} F^T\right]\mathcal{N}(0_{N(\Theta)},\II(\vt^*)) =: \mathbf{N}_F.
\end{align*}
Since by the chain rule $\HH(\vt_0)=F^T\HH(\vt^*)F$ the random vector $\mathbf{N}_F$ is distributed as \linebreak $\mathcal{N}(0_{N(\Theta)},\mathcal{M}_F(\vt^*)\II(\vt^*)\mathcal{M}_F(\vt^*))$ (note that $\mathcal{M}_F(\vt^*)$ is symmetric).
Finally, by  \eqref{proofcons3}, \Cref{helpprop}(d) and $C(n) \to \infty$ as $n \to \infty$,
\begin{eqnarray*}\label{proofcons12}
\lefteqn{\P ( \text{IC}_n(\Theta_{0}) - \text{IC}_n(\Theta_{}) < 0 )}\nonumber \\
&\stackrel{}{=}& \P \left( \frac12 \sqrt{n} \left( \widehat{\vt}^{n}_{} - f(\wh{\vt}^{n}_0) \right)^T  \nabla^2_\vt \wh\LL \left( \overline{\vt}^n, Y^{n} \right) \sqrt{n} \left( \widehat{\vt}^{n}_{} - f(\wh{\vt}^{n}_0) \right)
< - [N(\Theta_0) - N(\Theta)] C(n) \right) \nonumber\\
&\overset{n \to \infty}{\to}& \P \left( \mathbf{N}_F^T \HH(\vt^*) \mathbf{N}_F < \infty \right).
\end{eqnarray*}
Using \cite[Eq. (1.1)]{imhof} gives
%\begin{equation*}\label{proofcons13}
$\mathbf{N}_F^T \HH(\vt^*) \mathbf{N}_F \stackrel{\DD}{=} \sum_{i=1}^{N(\Theta)} \lambda_i \chi^2_i,$
%\end{equation*}
where $(\chi^2_i)$ is a sequence of independent $\chi^2$ random variables with one degree of freedom and
the $\lambda_i$ are the eigenvalues of \linebreak $\HH(\vt^*)^{\frac12} \MM_F(\vt^*) \II(\vt^*) \MM_F(\vt^*) \HH(\vt^*)^{\frac12}$. Since
$\rank(\mathcal{M}_F(\vt^*))=N(\Theta)-N(\Theta_0)$ and $\HH(\vt^*)^{\frac12}$ and $\II(\vt^*)$ have full rank, the number of strictly
positive eigenvalues of \linebreak $\HH(\vt^*)^{\frac12} \MM_F(\vt^*) \II(\vt^*) \MM_F(\vt^*) \HH(\vt^*)^{\frac12}$ is $N(\Theta)-N(\Theta_0)$.
Hence, the result follows.\\
\textsl{(c)} \, With the arguments in (b) we obtain the statement.
\end{proof}

\begin{remark}\label{consistremark} $\mbox{}$\\
%\begin{itemize}
%\item[(a)]
 (a) \, A conclusion of \Cref{strongcons}(a) is that strong consistency of the information criterion always holds,  independent of the process $Y$ generating the observed data and hence
  $\vartheta_0^E$, if \linebreak $\limsup_{n \to \infty} C(n)/\log(\log(n)) = \infty$. \\%    This is also true if we do not require that $\Theta_0$ is nested into $\Theta$ by a linear map, but rather allow $f$ to be a more general function that is injective and continuously differentiable.\\
(b)\, Let $\Theta_0$ be nested in $\Theta$ with map $F$. %$F: \R^{N(\Theta_0)} \to \R^{N(\Theta)}$ such that $F^TF =I_{N(\Theta_0) \times N(\Theta_0)}$ and $\Theta_0 = F^T \Theta$.
Then it can be shown as in the proof of \Cref{LLasymp} that
\begin{eqnarray*}
\lefteqn{\limsup_{n \to \infty}\frac{n}{\log(\log(n))}\left(\text{IC}_n(\Theta_{0}) - \text{IC}_n(\Theta_{})\right)}\\
&&= \lambda_{\text{max}}( \MM_F(\vt^\ast)^{\frac12} \II(\vt^\ast) \MM_F(\vt^\ast)^{\frac12})  + \limsup_{n \to \infty} [N (\Theta_0) - N(\Theta)] \frac{C(n)}{\log(\log(n))}.
\end{eqnarray*}
This implies that the information criterion $\text{IC}_n$ is not strongly consistent iff  \linebreak
$\limsup_{n \to \infty} C(n)/\log(\log(n)) < C^*$, where
%there exists a space $\Theta$ such that
%$$ \limsup_{n \to \infty} \frac{C(n)}{\log(\log(n))} < \frac{\lambda_{\text{max}}( \MM_F^{\frac12}(\vt^\ast) \II(\vt^\ast) \MM_F^{\frac12}(\vt^\ast))}{N(\Theta)-N(\Theta_0)}.$$
%Therefore, if one could calculate
$$C^\ast := \max_F \frac{\lambda_{\text{max}}( \MM_F(\vt^\ast)^{\frac12} \II(\vt^\ast) \MM_F(\vt^\ast)^{\frac12})}{N(\Theta)-N(\Theta_0)}
\leq \lambda_{\text{max}}(\HH(\vt^E_0)^{-\frac12} \II(\vt^E_0) \HH(\vt^E_0)^{-\frac12}).$$
Since the structure of $\HH(\vt^\ast)$ and $\II(\vt^\ast)$ is in general not known, it is difficult to calculate $C^*$ explicitly.
However, in the Gaussian case we will derive that $C^*=2$ (cf. \Cref{BMprop}). \\
(c) \, We would like to note that these results are similar to
the statement of \cite[Corollary 5.3]{sinwhite} under different model assumptions. However, the authors present only sufficient conditions
for strong consistency, where we also have a necessary condition (see \Cref{remarkwhite} as well).\\
(d) \, As the proof of \Cref{strongcons}(a), Case 1, shows, for spaces $\Theta$ with $\text{MCARMA}(A_{\vt},B_{\vt},C_{\vt},L_{\vt})\not=Y$  for every $\vt \in \Theta$ a necessary and sufficient condition for choosing the correct
parameter space asymptotically with probability 1 is   $\lim_{n \to \infty} C(n)/n = 0$. Only if we allow nested models as well the additional condition $\limsup_{n \to \infty} C(n)/\log(\log(n))>C^*$
becomes necessary. The probability in \Cref{thmcons1}(c) is the {overfitting probability}.
%\end{itemize}
\end{remark}

To wrap up this section, we want to study the special case where the observed MCARMA process
is driven by a Brownian motion. Some of the technical auxiliary results for the proof are given in the appendix.
\begin{corollary}\label{BMprop}
Assume that the L\'{e}vy process $L$ which drives the observed process $Y$ is a Brownian motion.
Then $IC_n$ is strongly consistent iff
$\limsup_{n \to \infty} C(n)/\log(\log(n))>2.$
%and it is not strongly consistent if
%$$\limsup_{n \to \infty} \frac{C(n)}{\log(\log(n))}<2.$$
%\item[(b)]If $\limsup C(n) = C < \infty$, then the overfitting probability of $IC_n$ for a space $\Theta$ in which $\Theta_0$ is nested is
%$$\P( \chi^2_{N(\Theta)-N(\Theta_0)} > [N(\Theta) - N(\Theta_0)] C ),$$
%where $ \chi^2_{N(\Theta)-N(\Theta_0)}$ denotes a $\chi^2$-distributed random variable with $N(\Theta) - N(\Theta_0)$ degrees of freedom.
\end{corollary}
\begin{proof}
From \Cref{BMlemma2}(b) we know that there exists a space ${\Theta}_0$ such that there is a  $\vt_0\in\Theta_0$ with
$\text{MCARMA}(A_{\vt_0},B_{\vt_0},C_{\vt_0},L_{\vt_0})=Y$ and
 ${\Theta}_0$ is nested in $\Theta_0^E$ with map $F$. Moreover, $N(\Theta_0)=N(\Theta_0^E)-1$ and
$$\lambda_{\text{max}}(\MM_F(\vt_0^E)^{\frac12} \II(\vt_0^E) \MM_F(\vt_0^E)^{\frac12}) = 2.$$
Additionally, a conclusion of \Cref{BMlemma2}(a) is that
 $$\lambda_{\text{max}}(\HH(\vt^E_0)^{-\frac12} \II(\vt^E_0) \HH(\vt^E_0)^{-\frac12})=2\lambda_{\max}(I_{N(\Theta_0^E)\times N(\Theta_0^E)})=2.$$
 Therefore the  statement  follows directly from \Cref{strongcons}(a) and \Cref{consistremark}(b).
%(b) \,
%\Cref{BMlemma2}(c) tells us that those eigenvalues of $\HH(\vt^*)^{\frac12} \MM_F(\vt^*) \II(\vt^*) \MM_F \HH(\vt^*)^{\frac12}$ can only be $0$ or $2$ where $2$ appears exactly %$N(\Theta) - N(\Theta_0)$ times. Hence,
%$$\sum_{i=1}^{N(\Theta)} \lambda_i \chi_i^2 = 2 \sum_{i=1}^{N(\Theta) - N(\Theta_0)}\chi^2_i \stackrel{\DD}{=} 2 \chi^2_{N(\Theta)-N(\Theta_0)}.$$
%The statement now follows from the definition of the overfitting probability in \Cref{thmcons1}(c).
\end{proof}
The results of this section are analogous to the ones obtained for ARMAX processes with i.i.d. noise  in \cite[Theorem 5.5.1]{hannandeistler}.

\section{AIC and BIC} \label{Section:5}
In this chapter, we transfer the two  most well-known information criteria, the AIC and BIC, to the MCARMA framework, highlight the main ideas in their development and
 apply the results of \Cref{sec:IC} to them.
\subsection{The Akaike Information Criterion (AIC)}\label{sec:AIC}
Historically, Akaike's idea was to study the Kullback-Leibler discrepancy of different models and choose the one which minimizes this quantity. In this section, we give arguments why this approach is also sensible in the case of MCARMA models.

As a starting point, let $g,f$ be probability densities on $\R^n$.  Then the {Kullback-Leibler discrepancy}
between $g$ and $f$ is
\begin{eqnarray*}
    K(g \mid f):=\int_{\R^n}f(x)\log\left(\frac{f(x)}{g(x)}\right)\, dx
     %   &=&\E_f\left[\log\left(\frac{f}{g}\right)\right]  \\
        =\E_f[\log\left({f}\right)]-\E_f[\log\left(g\right)]
        \geq 0.
\end{eqnarray*}
Equality holds only for $g=f$ (cf. \cite[p. 302]{brockwell}).
Let now $(f_\vt)_{\vt\in\Theta}$ be a family of densities on $\R^n$ and fix one ``true'' density $f_{\vt_0}$. With $\E_{\vt_0}$ we denote
the expectation regarding the distribution with density $f_{\vartheta_0}$.
Then, the density that comes closest to $f_{\vt_0}$ in the Kullback-Leibler sense is given by the one associated to
\begin{eqnarray*}
    \argmin_{\vartheta\in\Theta}K(f_\vt|f_{\vt_0})=\argmin_{\vartheta\in\Theta}\{\E_{\vt_0}[\log\left({f_{\vt_0}}\right)]-\E_{\vt_0}[\log\left(f_{\vt}\right)]\}
      %  &=&\argmin_{\vartheta\in\Theta}[-\E_{\vt_0}[\log\left(f_{\vt}\right)]] \\
        =\argmin_{\vartheta\in\Theta} \left\{ -\frac{2}{n}\E_{\vt_0}[\log\left(f_{\vt}\right)] \right\}.
\end{eqnarray*}
In our context $f_{\vt}$ denotes the density of the observations $Y^n$. The problem is that the right-hand side is not directly calculable so that we have to approximate it. To this end, let $\YY^n$ be an independent copy of
$Y^n$ and $\widehat{\vt}^n(Y^n)$ be the QMLE in $\Theta$ based on the observation $Y^n$. Then we use the approximation
\begin{eqnarray} \label{5.1}
     \min_{\vartheta\in\Theta} \left[ -\frac{2}{n}\E_{\vt_0}[\log\left(f_{\vt}\right)] \right]&\approx& -\frac{2}{n}\E_{\vt_0}[\log(f_{\widehat{\vt}^n(Y^n)}) \mid Y^n]
     =-\frac{2}{n}\E[\log(f_{\widehat{\vt}^n(Y^n)}(\YY^n)) \mid Y^n]\nonumber\\
     &\approx& \E\left[ \wh\LL(\widehat{\vt}^n(Y^n), \YY^n) \mid Y^n\right].
\end{eqnarray}
The right-hand side can again be approximated by the following theorem:
\begin{theorem} \label{Theorem 5.1}
Assume that the space $\Theta$ with associated family of continuous-time state space models $(A_{\vt}, B_{\vt}, C_{\vt}, L_{\vt})_{\vt \in \Theta}$ satisfies \autoref{as_D}. Then, as $n\to\infty$,
\begin{eqnarray*}
n\left(\wh\LL(\widehat{\vt}^n(Y^n), \YY^n)  - \left[\wh\LL(\widehat{\vt}^n(\YY^n), \YY^n)- \frac{\tr \left( \II(\vt^*) \HH^{-1}(\vt^*) \right)}{n}\right]\right)
    \stackrel{\DD}{\to} Z_{\vt^*},
\end{eqnarray*}
where $Z_{\vt^*}$ is a random variable with expectation $\E[Z_{\vt^*}]=0$. In particular, as $n\to\infty$,
\begin{eqnarray*}
\wh\LL(\widehat{\vt}^n(Y^n), \YY^n)  - \left[\wh\LL(\widehat{\vt}^n(\YY^n), \YY^n)- \frac{\tr \left( \II(\vt^*) \HH^{-1}(\vt^*)\right)}{n}\right]
    \stackrel{\P}{\to} 0.
\end{eqnarray*}
\end{theorem}
\begin{proof}
A second-order Taylor expansion of $\wh\LL(\widehat{\vt}^n(\YY^n), Y^n)$ around $\widehat{\vt}^n(Y^n)$ gives
\begin{eqnarray*}
\wh \LL(\widehat{\vt}^n(\YY^n), Y^n)= \wh\LL(\widehat{\vt}^n(Y^n), Y^n) %+ \nabla_\vt \wh\LL (\widehat{\vt}^n(Y^n_{\vt_0}),Y^n_{\vt_0})
%\left( \widehat{\vt}^n(\YY^n_{\vt_0}) - \widehat{\vt}^n(Y^n_{\vt_0}) \right) \\
                                             + \frac{1}{2} \left( \widehat{\vt}^n(\YY^n) - \widehat{\vt}^n(Y^n) \right)^T \nabla^2_\vt \wh\LL (\overline{\vt}^n,Y^n) \left( \widehat{\vt}^n(\YY^n) - \widehat{\vt}^n(Y^n) \right),
\end{eqnarray*}
 where $\|\overline{\vt}^n-\widehat{\vt}^n(Y^n)\|\leq \|\widehat{\vt}^n(\YY^n)-\widehat{\vt}^n(Y^n)\|$.
Hence,
\begin{eqnarray*}
    \lefteqn{\wh\LL(\widehat{\vt}^n(\YY^n), Y^n)   - \wh\LL(\widehat{\vt}^n(Y^n), Y^n)}\\
        &&=\frac{1}{2}\text{tr} \left( \nabla^2_\vt \wh\LL (\overline{\vt}^n,Y^n)\left(\widehat{\vt}^n(\YY^n) - \widehat{\vt}^n(Y^n) \right)  \left( \widehat{\vt}^n(\YY^n) - \widehat{\vt}^n(Y^n)\right)^T \right).
\end{eqnarray*}
On the one hand, since both $\widehat{\vt}^n(Y^n)$ and $ \widehat{\vt}^n(\YY^n)$ converge $\p$-a.s. to $\vt^*$,
the vector $\overline{\vt}^n\to\vt^*$ $\p$-a.s. as well. On the other hand, by the independence of
 $Y^n$ and $\YY^n$, the random vectors $\widehat{\vt}^n(\YY^n)$
 and $\widehat{\vt}^n(Y^n)$ are independent as well.   By \Cref{idprop}, as
 $n\to\infty$,
\begin{eqnarray*}
 \sqrt{n} \left(\widehat{\vt}^n(Y^n)- \vt^*, \widehat{\vt}^n(\YY^n)-\vt^*\right) \stackrel{\DD}{\to} (\mathcal{N}_1,\mathcal{N}_2),
 \end{eqnarray*}
 where $\mathcal{N}_1,\mathcal{N}_2$ are independent, $\NN(0_{N(\Theta)}, {\HH}^{-1}(\vt^*) \II(\vt^*) {\HH}^{-1}(\vt^*))$-distributed random vectors. A conclusion of \Cref{helpprop}(d) is
$\nabla^2_\vt \wh\LL( \overline{\vt}^n, Y^n ) %= \frac{1}{n} \sum_{k=1}^{n} \nabla^2_\vt l_{\vt, k}
 {\to} \HH(\vt^*)$ $ \p\text{-a.s.}$ Hence,
a continuous mapping theorem gives
 \begin{eqnarray*}
    n\left(\wh\LL(\widehat{\vt}^n(\YY^n), Y^n)   - \wh\LL(\widehat{\vt}^n(Y^n), Y^n)\right)\stackrel{\DD}{\to}
    \frac{1}{2}\tr\left(\HH(\vt^*)(\mathcal{N}_1+\mathcal{N}_2)(\mathcal{N}_1+\mathcal{N}_2)^T\right),
 \end{eqnarray*}
 and by the independence of $\mathcal{N}_1$ and $\mathcal{N}_2$ we have
 \begin{eqnarray*}
     \E\left[\HH(\vt^*)(\mathcal{N}_1+\mathcal{N}_2)(\mathcal{N}_1+\mathcal{N}_2)^T\right]
     =2\HH(\vt^*)\E\left[\mathcal{N}_1\mathcal{N}_1^T\right]=2\II(\vt^*) {\HH}^{-1}(\vt^*).
 \end{eqnarray*}
The statement follows then obviously since the expectation of the trace is the trace of the expectation.
\end{proof}

As a consequence of \eqref{5.1} and \Cref{Theorem 5.1} we receive the approximation
\begin{eqnarray*}
    \min_{\vartheta\in\Theta} \left[ -\frac{2}{n}\E_{\vt_0}[\log\left(f_{\vt}\right)] \right]\approx \wh\LL(\wh \vt^n(\YY^n),\YY^n)+\frac{\tr \left( \II(\vt^*) \HH^{-1}(\vt^*)\right)}{n},
\end{eqnarray*}
which becomes our information criterion via  the following definition:
\begin{definition}\label{aicdef}
For a space $\Theta$ with associated family of continuous-time state space models \linebreak $(A_{\vt}, B_{\vt}, C_{\vt}, L_{\vt})_{\vt \in \Theta}$ that satisfies \autoref{as_D}, the
\textsl{Akaike Information Criterion (AIC)} is defined as
\begin{eqnarray*}
    AIC_n(\Theta)=\wh\LL(\wh \vt^n,Y^n)+\frac{\tr \left( \II(\vt^*) \HH^{-1}(\vt^*)\right)}{n}.
\end{eqnarray*}
\end{definition}
In general $\II(\vt^*)$ and $\HH(\vt^*)$ are not known.
For practical purposes,  they have to be estimated. For both, estimators are known and can be found at the end of \cite[Section 2.2]{schlemmstelzer}, for example.

\begin{remark} \label{remark AIC Brownian}
If the L\'{e}vy process $L$ which drives the observed process $Y$ is a Brownian motion and $\text{MCARMA}(A_{\vt^\ast}, B_{\vt^\ast}, C_{\vt^\ast}, L_{\vt^\ast}) = Y$, we have
 $ \II(\vt^*)= 2\HH(\vt^*)$ by \Cref{BMlemma2} and hence, the AIC  reduces to
\begin{eqnarray*}
    AIC_n(\Theta)=\wh\LL(\wh \vt^n,Y^n)+\frac{2 N(\Theta)}{n}.
\end{eqnarray*}
\end{remark}
The form of the AIC given in this remark coincides with Akaike's original definition (cf. \cite{akaike}). For these reasons, it suggests itself to define an alternative version of the AIC as follows:
\begin{definition}\label{caicdef}
For a space $\Theta$ with associated family of continuous-time state space models \linebreak $(A_{\vt}, B_{\vt}, C_{\vt}, L_{\vt})_{\vt \in \Theta}$ that satisfies \autoref{as_D}, the
\textsl{Classical Akaike Information Criterion (CAIC)} is defined as
\begin{eqnarray*}
    CAIC_n(\Theta)=\wh\LL(\wh \vt^n,Y^n)+\frac{2 N(\Theta)}{n}.
\end{eqnarray*}
\end{definition}
This criterion avoids the additional work of estimating the matrices  $ \II(\vt^*)$ and $\HH^{-1}(\vt^*)$ appearing in the AIC, which comes at the cost of not being exact when the driving L\'{e}vy process is not a Brownian motion. For both versions of the AIC, we can immediately make a statement about consistency:
\begin{theorem}
Both the AIC and the CAIC are neither strongly nor weakly consistent.
\end{theorem}
\begin{proof}
The CAIC is a special case of $IC_n$ with $C(n) = 2$ such that the assertion follows from \Cref{strongcons}(b). For the AIC, the proof of \Cref{strongcons}(b) can be directly adapted.
\end{proof}

\subsection{The Bayesian Information Criterion (BIC)}
Another information criterion which appears often in the literature is the so-called \emph{Bayesian Information Criterion} (BIC), sometimes also called SIC, an abbreviation for Schwarz Information Criterion, named after the author who originally introduced it in \cite{Schwarz:1978}. Another often-cited article in this context is \cite{rissanen}, which introduces an equivalent criterion in a slightly different context based on coding theory.
As the name Bayesian Information Criterion already suggests, the approach of the definition is based on  Bayesian statistics.
Our derivation is based on \cite{cavanaughneath}, relying on properties of the likelihood function. Suppose that $\pi$ is a discrete prior probability distribution over the set of candidate spaces $\Theta$ and $\pi(\Theta)>0$ for every parameter space $\Theta$ which will be considered. Moreover, suppose that $g( \cdot \mid \Theta)$ is a prior probability distribution over the parameter space $\Theta$. For $g$ we require the following assumption.
\begin{assumptionletter}
         \label{as_C}
         \renewcommand{\theenumi}{C.\arabic{enumi}}
         \renewcommand{\labelenumi}{\theenumi}
         For every space $\Theta$ there exist two constants $b$ and $B$ with $0 < b \leq B < \infty$ such that $0 \leq g(\vt \mid \Theta) \leq B$ for all $\vt \in \Theta$ and $b \leq g(\vt \mid \Theta)$ for all $\vt$ in some neighborhood of the pseudo-true parameter $\vt^*\in\Theta$.
\end{assumptionletter}
Now we can apply Bayes' theorem to obtain the joint posterior probability distribution $f$ of $\Theta$ and $\vt$ which is
\begin{equation}\label{posterior1}
f( \Theta, \vt \mid  Y^n ) = \frac{ \pi(\Theta) g (\vt \mid \Theta) f(  Y^n \mid \Theta, \vt)}{h( Y^n)},
\end{equation}
where $h( \cdot )$ denotes the (unknown) marginal density of $ Y^n$. With this, we can calculate the a posteriori probability of space $\Theta$ as
\begin{equation}\label{probm}
\P (\Theta \mid  Y^n) = \int_{\Theta} f( \Theta, \vt \mid  Y^n) \dif \vt.
\end{equation}
The idea is to choose the most probable model for the data at hand, i. e. the space $\Theta$ which maximizes the a posteriori probability. Similar to the derivation of the AIC, the task is now to find a good approximation of $\eqref{probm}$ which is directly calculable from the data.
For this note first that maximization of \eqref{probm} is equivalent to minimizing $-2/n$ times the logarithm of $\P (\Theta \mid  Y^n)$. Applying this transformation and plugging in \eqref{posterior1} gives
\begin{align}
-\frac{2}{n} \log \left( \P ( \Theta \mid  Y^n ) \right) = \frac{2}{n} \log ( h (  Y^n)) - \frac{2}{n} \log ( \pi (\Theta))
-\frac{2}{n} \log \left( \int_{\Theta} f(  Y^n \mid \Theta, \vt) g(\vt \mid \Theta) \dif \vt \right). \label{logprob}
\end{align}
We choose the parameter space $\Theta$ with the lowest value of $-\frac{2}{n} \log \left( \P ( \Theta \mid  Y^n ) \right)$.
Hence, we have to approximate this expression. For this, we approximate the unknown density $f( Y^n \mid \Theta, \vt)$ by the pseudo-Gaussian likelihood function $\cancel{L}(\vt, Y^n) = \exp( -\frac{n}{2} \widehat{\LL}(\vt, Y^n) )$ and use the following theorem.

% by definition we have $f(  Y^n \mid \Theta, \vt) = \cancel{L}(\vt \mid  Y^n)$, where $\cancel{L}(\vt \mid  Y^n)$ denotes the pseudo-Gaussian likelihood of $\vt$ given $ Y^n$. By means of this and the posterior distribution we can write the left-hand side of $\eqref{probm}$ as
% \begin{equation*}\label{probm2}
% \P (\Theta \mid  Y^n) = {h( Y^n)}^{-1} \pi (\Theta) \int_{\Theta} \cancel{L}(\vt \mid  Y^n) g(\vt \mid \Theta) \dif \vt.
% \end{equation*}
% Instead of maximizing this quantity, we can also look for the minimum of

\begin{theorem}
Assume that the space $\Theta$ with associated family of continuous-time state space models $(A_{\vt}, B_{\vt}, C_{\vt}, L_{\vt})_{\vt \in \Theta}$ satisfies \autoref{as_D}
and the a priori density $g$ satisfies \autoref{as_C}.
Then
\begin{align*}
\wh\LL( \widehat{\vt}^n, Y^n) + N(\Theta) \frac{\log (n)}{n} + \frac{R_1 (N(\Theta))}{n} &\leq -\frac{2}{n} \log \left( \int_\Theta \cancel{L}(\vt, Y^n) g(\vt \mid \Theta) \dif \vt \right) \notag \\
                                                                       &\leq \wh\LL( \widehat{\vt}^n,  Y^n) + N(\Theta) \frac{\log (n)}{n} +
                                                                       \frac{R_2(N(\Theta))}{n}, \label{BICineq}
\end{align*}
where $R_1(N(\Theta))$ and $R_2(N(\Theta))$ are rest terms which do not depend on $n$. In particular,
\begin{eqnarray*}
    -\frac{2}{n} \log \left( \P ( \Theta \mid  Y^n ) \right) =\wh\LL( \widehat{\vt}^n,  Y^n) + N(\Theta) \frac{\log (n)}{n}+
    \left[\frac{2}{n} \log ( h (  Y^n))+O\left(\frac{\log(n)}{n}\right)\right].
\end{eqnarray*}
\end{theorem}
\begin{proof}
By \autoref{as_D}, \autoref{as_C}, \Cref{helpprop} and \cite[Proposition 3.1]{sinwhite}
the regularity assumptions in  \cite{cavanaughneath} are satisfied so that the statement follows from there.
\end{proof}
The term $\frac{2}{n} \log ( h (  Y^n))$  is the same across all parameter spaces and therefore not relevant for model selection.
Based on these ideas, we define the BIC.

\begin{definition}\label{bicdef}
Assume that the space $\Theta$ with associated family of continuous-time state space models $(A_{\vt}, B_{\vt}, C_{\vt}, L_{\vt})_{\vt \in \Theta}$ satisfies \autoref{as_D}.
Then the \textsl{Bayesian Information Criterion (BIC)} is defined as
\begin{equation*}\label{BICdef}
\text{BIC}_n(\Theta) := \wh \LL( \widehat{\vt}^n,  Y^n) + N(\Theta) \frac{\log (n)}{n}.
\end{equation*}
\end{definition}
As with the AIC, we can immediately make a statement about consistency of the BIC:
\begin{theorem}
The BIC is a strongly consistent information criterion.
\end{theorem}
\begin{proof}
The BIC is a special case of $IC_n$ with $C(n) = \log(n)$. The assertion immediately follows from \Cref{strongcons}(a), since $\lim_{n \to \infty} {\log(n)/\log(\log(n))} = \infty$ (see also \Cref{consistremark}(a)).
\end{proof}

\section{Simulation study}\label{sec:simulation}
The results on information criteria obtained in the previous sections will now be illustrated by  a simulation study.
In this context we would like to thank Eckhard Schlemm and Robert Stelzer who kindly provided  the MATLAB code for the simulation and parameter estimation
of the MCARMA process.
 As before, we use the
Echelon MCARMA parametrization in the simulations. We simulate a two-dimensional MCARMA process with Kronecker index $m_0 = (1, 2)$ for two parameter values.
One is an  MCARMA(2,0) process with parameter
$$\vt^{(1)}_0 = \begin{pmatrix} -1 & -2 & 1 & -2 & -3 & 0 & 0 \end{pmatrix}.$$
The other is an MCARMA(2,1) process with parameter
$$\vt^{(2)}_0 = \begin{pmatrix} -1 & -2 & 1 & -2 & -3 & 1 & 2 \end{pmatrix}.$$
As  driving L\'{e}vy process, we use, on the one hand,  a two-dimensional, correlated Brownian motion
and, on the other hand, a two-dimensional, normal-inverse Gaussian (NIG) process.
For the NIG process the increments $L(t) - L(t-1)$ have the density
$$f_{NIG}(x; \mu, \alpha, \beta, \delta, \Delta) = \frac{\delta \e^{\delta \kappa}}{2 \pi} \frac{\e^{ \langle \beta x \rangle}}{\e^{\alpha g(x)}} \frac{1 + \alpha g(x)}{g(x)^3}, \quad x \in \R^2,$$
where
$$g(x) = \sqrt{\delta^2 + \langle x- \mu, \Delta ( x - \mu )\rangle}, \;\quad \kappa^2 = \alpha^2 - \langle \beta, \Delta \beta \rangle.$$
The parameter $\mu \in \R^2$ is a location parameter, $\alpha \geq 0$ is a shape parameter, $\beta \in \R^2$ is a symmetry parameter, $\delta \geq 0$ is a scale parameter and $\Delta \in \R^{2 \times 2}$ is a positive semidefinite matrix with $\det(\Delta) = 1$ that determines the dependence between the components of the L\'{e}vy process. In the simulations we use the values
$$\delta = 1, \;\quad \alpha = 3, \;\quad \beta= \begin{pmatrix} 1 \\ 1 \end{pmatrix},\quad \Delta = \begin{pmatrix} \frac54 & -\frac12 \\ -\frac12 & 1 \end{pmatrix},\quad \mu = -\frac{1}{2\sqrt{31}} \begin{pmatrix} 3 \\ 2 \end{pmatrix},$$
which result in a zero-mean process with covariance matrix
$$\Sigma^{L}_{NIG} \approx \begin{pmatrix} 0.4571 & -0.1622 \\ -0.1622 & 0.3708 \end{pmatrix}.$$
In the case of the Brownian motion the covariance matrix $\Sigma^L_{BM}$ is equal to the covariance matrix $\Sigma^{L}_{NIG}$ in the NIG case.
In the estimation the number of free parameters includes three parameters for the covariance matrix of the driving L\'{e}vy process.\\
  The simulation of the continuous-time process is done with the initial value $X(0) = 0$, applying the Euler-Maruyama method to the stochastic differential equation \eqref{cstateeq} and then evoking \eqref{CARMA:observation}. For the Euler-Maruyama scheme we operate on the interval $[0, 2000]$ and take the step size  $0.01$. Afterwards, the simulated process is sampled at discrete points in time with sampling distance $h=1$, resulting in $n=2000$ observations.
After obtaining the discrete samples of the MCARMA process we calculate the  AIC, CAIC and BIC as defined in \Cref{aicdef}, \Cref{caicdef} and \Cref{bicdef}, respectively.
In the calculation of the AIC we estimate the penalty term $\tr \left( \II(\vt^*) \HH^{-1}(\vt^*)\right)$ by the methods
presented in \cite[Section 2.2]{schlemmstelzer} as well
since in general there is no explicit form of $\II(\vt^*)$ and $\HH(\vt^*)$.
We consider eight different parameter spaces in total. While some of them differ in the Kronecker index, others differ only by the degree of the MA polynomial of the MCARMA process.
We compare the different values of the information criteria  and write down the space for which the minimum values is attained.
The results of $50$ replications are summarized in \Cref{Table1}.

\begin{table}[h]
% \small
\centering
\begin{tabular}{c||c|c|c|c||c|c|c||c|c|c|c}
 \hline
 Space & \multicolumn{4}{c||}{Model} & \multicolumn{3}{c||}{BM} & \multicolumn{3}{c|}{NIG} \\
 \hline
    & $m$ & $p$ & $q$ & $N(\Theta)$ & AIC & CAIC & BIC & AIC & CAIC  & BIC \\
 \hline
  1& $(1,1)$ & $1$ & $0$ & 7 & 0 & 0  & 0 & 0 & 0  & 0 \\
 \hline
  2 & $(1,2)$ & $2$ & $1$ & 10 & 14 & 8  & 1 & 10 & 4  & 0 \\
  \hline
  3 & $(1,2)$ & $2$ & $0$ & 8 & 36 & 42  & 49 & 40 & 46  & 50 \\
    \hline
  4 & $(2,1)$ & $2$ & $1$ & 11 & 0 & 0  & 0 & 0 & 0  & 0 \\
    \hline
  5 & $(2,1)$ & $2$ & $0$ &  9 & 0 & 0  & 0 & 0 & 0  & 0 \\
    \hline
  6 & $(2,2)$ & $2$ & $1$ & 15 & 0 & 0  & 0 & 0 & 0  & 0 \\
    \hline
  7 & $(2,2)$ & $2$ & $0$ & 11 & 0 & 0  & 0 & 0 & 0  & 0 \\
    \hline
  8 & $(3,2)$ & $3$ & $2$ & 19 & 0 & 0  & 0 & 0 & 0  & 0 \\
  \hline
\end{tabular}
\caption{Results for the true parameter $\vt^{(1)}_0$ and $\Sigma^L_{BM} = \Sigma^{L}_{NIG}$.} \label{Table1}
\end{table}

 As expected because of the strong consistency the BIC performs convincingly and has a high accuracy in both cases. It even achieves a perfect score in the case where the driving noise is a NIG process and makes one wrong decision in the BM scenario. Furthermore, both versions of the AIC exhibit  overfitting. There is an undeniable difference between the CAIC and the AIC in both cases. From the theory, we know that this should not happen when the driving L\'{e}vy process is a Brownian motion since the criteria are
 then the same. This difference comes from the estimation error by estimating
 the penalty term $\tr \left( \II(\vt^*) \HH^{-1}(\vt^*)\right)$ in the AIC. We realize that in the Gaussian model the estimation error of the penalty term is usually higher for model number 3 than for model 2 (relative to the true values), which results in a higher overfitting rate for the AIC.
% We realized that in the Gaussian model we usually underestimated the penalty term which
% results in a higher overfitting.\\
We also calculate the overfitting probability in the Brownian motion case as given in \Cref{strongcons}(c). For this, note that there is only one parameter space in which the true one is nested (space number $2$) and for that space we have $C=2$, $N(\Theta) - N(\Theta_0) = 2$ and the entries of $F$ are given by
$$
F_{ij}= \begin{cases} 1, &\text{if  } i=j \text{  and  } i \in \{1, 2, 3, 4, 5, 8, 9, 10 \}, \\
                      0, &\text{otherwise.}
\end{cases}$$
% $$F= \begin{pmatrix} 1 & 0 &0 &0 &0 & 0 &0 &0 \\ 0 & 1 &0 &0 &0 & 0 &0 &0 \\ 0 & 0 &1 &0 &0 & 0 &0 &0 \\0 & 0 &0 &1 &0 & 0 &0 &0 \\ 0 & 0 &0 &0 &1 & 0 &0 &0 \\ 0 & 0 &0 &0 &0 & 0 &0 &0 \\ 0 & 0 &0 &0 &0 & 0 &0 &0 \\ 0 & 0 &0 &0 &0 & 1 &0 &0 \\ 0 & 0 &0 &0 &0 & 0 &1 &0 \\ 0 & 0 &0 &0 &0 & 0 &0 &1 \end{pmatrix}$$
The strictly positive eigenvalues of $\HH(\vt^*)^{\frac12} \MM_F(\vt^*) \II(\vt^*) \MM_F(\vt^*) \HH(\vt^*)^{\frac12}$ are calculated with the help of MATLAB and turn out to be both equal to $2$, so that the overfitting probability simplifies to
$$\P( \chi^2_1 > 2 ) \approx 0.1573.$$
The empirical probability $8/50 = 0.16$ of overfitting in the CAIC  is very close.
The results of the simulation study for $\vt^{(2)}_0$ are given in \Cref{Table:2}.\\

\begin{table}[h]
% \small
\centering
\begin{tabular}{c||c|c|c|c||c|c|c||c|c|c}
 \hline
 Space & \multicolumn{4}{c||}{Model} & \multicolumn{3}{c||}{BM} & \multicolumn{3}{c}{NIG} \\
 \hline
    & $m$ & $p$ & $q$ & $N(\Theta)$ & AIC & CAIC  & BIC & AIC & CAIC & BIC \\
 \hline
  1& $(1,1)$ & $1$ & $0$ & 7 & 0 & 0 & 0  & 0 & 0  & 0 \\
 \hline
  2 & $(1,2)$ & $2$ & $1$ & 10 & 50 & 50  & 50 & 50 &  50  & 50 \\
  \hline
  3 & $(1,2)$ & $2$ & $0$ & 8 &0 & 0  & 0 & 0 & 0  & 0 \\
    \hline
  4 & $(2,1)$ & $2$ & $1$ & 11& 0 & 0  & 0 & 0 & 0  & 0 \\
    \hline
  5 & $(2,1)$ & $2$ & $0$ & 9& 0 & 0  & 0 & 0 & 0  & 0 \\
    \hline
  6 & $(2,2)$ & $2$ & $1$ & 15 &0 & 0  & 0 & 0 & 0  & 0 \\
    \hline
  7 & $(2,2)$ & $2$ & $0$ & 11& 0 & 0  & 0 & 0 & 0  & 0 \\
    \hline
  8 & $(3,2)$ & $3$ & $2$ & 19& 0 & 0 & 0 & 0 & 0 & 0   \\
  \hline
\end{tabular}
\caption{Results for the true parameter $\vt^{(2)}_0$ and $\Sigma^L_{BM} = \Sigma^{L}_{NIG}$.} \label{Table:2}
\end{table}

As we can see all the information criteria  perform perfectly. There are no effects of overfitting, which is not surprising considering the fact that the true parameter is chosen in such a way that it is not contained in any of the other spaces besides space number 2, so that the scenario from \Cref{consistremark}(c) is given.
\appendix
\section{Appendix}

\subsection{Auxiliary results for \Cref{Section:3}} \label{Section:3.3}

We summarize  some auxiliary results which are used throughout the paper. We start with a lemma giving moving average representations of the pseudo-innovations and their derivatives.
\begin{lemma}\label{serieslemma}
Assume that the space $\Theta$ with associated family of continuous-time state space models $(A_{\vt}, B_{\vt}, C_{\vt}, L_{\vt})_{\vt \in \Theta}$ satisfies Assumptions $\ref{as_D1}$ to $\ref{as_D5}$.
\begin{itemize}
\item[(a)] There exists a matrix sequence $(c_{\vt, k})_{k \in \N}$ such that
\begin{equation*}\label{seriesrep}
\epsilon_{\vt,k} = Y(kh) + \sum_{\nu = 1}^{\infty} c_{\vt, \nu} Y((k-\nu)h), \quad k \in \Z.
\end{equation*}
Furthermore, there exists a positive constant $C$ and a constant $\rho \in (0, 1)$ such that
\begin{eqnarray*}
    \sup_{\vt \in \Theta} \| c_{\vt, k} \| \leq C \rho^k, \quad k \in \N.
\end{eqnarray*}
\item[(b)] For each $i \in \{ 1, \ldots, N(\Theta) \}$, there exists a matrix sequence $(c^{(i)}_{\vt, k})_{k \in \N}$ such that
\begin{equation*}\label{seriesrep2}
\partial_i \epsilon_{\vt, k} = \sum_{\nu = 1}^{\infty} c^{(i)}_{\vt, \nu} Y((k-\nu)h), \quad k \in \Z.
\end{equation*}
Furthermore, there exists a positive constant ${C}$ and a constant ${\rho} \in (0, 1)$ such that
\begin{eqnarray*}
    \sup_{\vt \in \Theta} \| c^{(i)}_{\vt, k} \| \leq {C} {\rho}^k, \quad k \in \N.
\end{eqnarray*}
\item[(c)] For each $i, j \in \{ 1, \ldots, N(\Theta) \}$, there exists a matrix sequence $(c^{(i,j)}_{\vt, k})_{k \in \N}$ such that
\begin{equation*}\label{seriesrep3}
\partial^2_{i,j} \epsilon_{\vt, k} = \sum_{\nu=1}^{\infty} c^{(i,j)}_{\vt, \nu} Y((k-\nu)h), \quad k \in \Z.
\end{equation*}
Furthermore, there exists a positive constant ${C}$ and a constant ${\rho} \in (0, 1)$ such that
\begin{eqnarray*}
    \sup_{\vt \in \Theta} \| c^{(i,j)}_{\vt, k} \| \leq {C} {\rho}^k, \quad  k \in \N.
\end{eqnarray*}
\end{itemize}
\end{lemma}
\begin{proof}
Part (a) is \cite[Lemma 2.6ii)]{schlemmstelzer}, part (b) is \cite[Lemma 2.11ii)]{schlemmstelzer} and part (c) is \cite[Lemma 2.11iv)]{schlemmstelzer}
where we additionally use \cite[Lemma 3.14]{schlemmstelzer}.
\end{proof}

In the next step, we show that it does not matter whether we consider the approximate pseudo-innovations or  the pseudo-innovations.
\begin{lemma}\label{helemma}
Assume that the space $\Theta$ with associated family of continuous-time state space models $(A_{\vt}, B_{\vt}, C_{\vt}, L_{\vt})_{\vt \in \Theta}$ satisfies Assumptions $\ref{as_D1}$ to $\ref{as_D5}$.
If for $i, j \in \{ 1, \ldots, N(\Theta) \}$ the initial values $\widehat{X}_{\vt,1}$ are such that $\sup_{\vt \in \Theta} \| \widehat{X}_{\vt,1} \|$, $\sup_{\vt \in \Theta} \| \partial_i \widehat{X}_{\vt,1} \|$ and  $\sup_{\vt \in \Theta} \| \partial^2_{i,j} \widehat{X}_{\vt,1} \|$ are almost surely finite, then it holds:
\begin{itemize}
\item[(a)] $\sup_{\vt \in \Theta} \left| \widehat{\LL} \left( \vt, Y^{n} \right) - \LL \left( \vt, Y^{n} \right) \right|\to 0$  as $n\to\infty$ $\Pas$
\item[(b)] $ \sqrt{n}\sup_{\vt \in \Theta} \left| \partial_i \widehat{\LL} \left( \vt, Y^{n} \right) - \partial_i \LL \left( \vt, Y^{n} \right) \right|\stackrel{\p}{\to}0$ as $n\to\infty$.
\item[(c)]$ \sup_{\vt \in \Theta} \left| \partial^2_{i,j} \widehat{\LL} \left( \vt, Y^{n} \right) - \partial^2_{i,j} \LL \left( \vt, Y^{n} \right) \right|\to 0$ as $n\to\infty$ $\Pas$
\item[(d)]$\sup_{\vt \in \Theta} \E \left[ \left| \wh \LL ( \vt, Y^n ) - \LL ( \vt, Y^n ) \right| \right] \to 0$ as $n \to \infty$.
\end{itemize}
\end{lemma}
\begin{proof}
(a) \, is \cite[Lemma 2.7]{schlemmstelzer} taking \cite[Lemma 3.14]{schlemmstelzer} into account. The proof of (b) and (c)
follows in the same way by using \cite[Lemma 2.11]{schlemmstelzer}.\\
(d) \, As in the proof of \cite[Lemma 2.7]{schlemmstelzer}, we have $\sup_{\vt \in \Theta} \E  \left\| \wh \epsilon_{\vt, k} \right\|  < \infty$, $\sup_{\vt \in \Theta} \E \left\| \epsilon_{\vt, k} \right\|  < \infty$ and for some $\rho \in (0,1)$ the behavior
\begin{eqnarray*}
\sup_{\vt \in \Theta} \E \left[ \left| \wh \LL ( \vt, Y^n ) - \LL ( \vt, Y^n ) \right| \right] %&\leq \sup_{\vt \in \Theta} \E \left[ \frac{C}{n} \sum_{k=1}^{n} \rho^k \left( \| \he_{\vt,k} \| + \| \epsilon_{\vt,k} \| \right) \right] \\
\leq  \frac{C}{n} \sum_{k=1}^{n} \rho^k \sup_{\vt \in \Theta} (\E  \| \he_{\vt,k} \|  + \E  \| \epsilon_{\vt,k} \|) \stackrel{n \to \infty}{\to} 0.
\end{eqnarray*}

\end{proof}
We conclude this section with another lemma, which plays a role in the proof of consistency of information criteria for MCARMA processes.
\begin{lemma}\label{loglikelemma}
Assume that the space $\Theta$ with associated family of continuous-time state space models $(A_{\vt}, B_{\vt}, C_{\vt}, L_{\vt})_{\vt \in \Theta}$ satisfies \autoref{as_D}.
 Let  $\text{MCARMA}(A_{\vt_0},B_{\vt_0},C_{\vt_0},L_{\vt_0})=Y$.
 Then for every $\vt \in \Theta$
\begin{eqnarray*}
\mathcal{Q}(\vartheta)-\mathcal{Q}(\vartheta_0)\geq \tr \left( V^{-1}_{\vt} \E \left[ \left( \epsilon_{\vt, 1} - \epsilon_{\vt_0, 1} \right) \left( \epsilon_{\vt, 1} - \epsilon_{\vt_0, 1} \right)^T \right]\right)\geq 0.
\end{eqnarray*}
Furthermore, if $\text{MCARMA}(A_{\vt},B_{\vt},C_{\vt},L_{\vt})\not=Y$, then
\begin{equation}\label{loglikelemmaeq2}
\tr \left( V^{-1}_{\vt} \E \left[ \left( \epsilon_{\vt, 1} - \epsilon_{\vt_0, 1} \right) \left( \epsilon_{\vt, 1} - \epsilon_{\vt_0, 1} \right)^T \right]\right)>0.
\end{equation}
\end{lemma}
\begin{proof}
The proof is given in \cite[Lemma~2.10]{schlemmstelzer}.
\end{proof}

\subsection{Auxiliary results for \Cref{sec:IC}}

In this appendix, we give the calculations for the Brownian motion case in \Cref{sec:IC}.
\begin{lemma}\label{tracelemma}
Let $A, B \in \R^{d \times d}$ be matrices, where $B$ is symmetric. Then
$$\tr \left( (\vecc(I_{d \times d}) \otimes \vecc( I_{d \times d} )^T) (A \otimes B ) \right) = \tr ( AB).$$
\end{lemma}
\begin{proof} The proof can be derived by straightforward algebraic calculations.
\end{proof}

\begin{lemma}\label{BMlemma2}
Assume that the L\'{e}vy process $L$ which drives the observed process $Y$ is a Brownian motion.
\begin{itemize}
\item[(a)] Assume that the space $\Theta$ with associated family of continuous-time state space models \linebreak $(A_{\vt}, B_{\vt}, C_{\vt}, L_{\vt})_{\vt \in \Theta}$ satisfies \autoref{as_D} and that $\text{MCARMA}(A_{\vt^\ast},B_{\vt^\ast},C_{\vt^\ast},L_{\vt^\ast})=Y$ for the pseudo-true parameter $\vt^\ast$. Then
$$\II(\vt^\ast)=2\HH(\vt^\ast).$$
\item[(b)] There exists a space ${\Theta}_0$  with associated family of continuous-time state space models \linebreak $(A_{\vt}, B_{\vt}, C_{\vt}, L_{\vt})_{\vt \in {\Theta_0}}$ satisfying \autoref{as_D} such that $\text{MCARMA}(A_{\vt_0},B_{\vt_0},C_{\vt_0},L_{\vt_0})=Y$ for some $\vt_0\in\Theta_0$. Moreover, $\Theta_0$ is
% with associated family of continuous-time state space models \linebreak $(A_{\vt}, B_{\vt}, C_{\vt}, L_{\vt})_{\vt \in \widetilde{\Theta}}$
nested in $\Theta_0^E$ with map $F$, $N({\Theta}_0)=N(\Theta_0^E)-1$ and
$$\lambda_{\text{max}}(\MM_F(\vt_0^E)^{\frac12} \II(\vt_0^E) \MM_F(\vt_0^E)^{\frac12}) = 2.$$
%\item[(c)] The matrix $\HH(\vt^*)^{\frac12} \MM_F \II(\vt^*) \MM_F \HH(\vt^*)^{\frac12}$ has only the eigenvalues $2$ and $0$, where $0$ has multiplicity $N(\Theta_0)$ and $2$ has multiplicity $N(\Theta) - N(\Theta_0)$.
\end{itemize}
\end{lemma}
\begin{proof}
(a) \, An analogous statement for vector ARMA processes is given in \cite[Remark 2]{boubacar}. However, they state it without a proof.
Since the proof is not so obvious we decided to sketch it here for MCARMA processes.
First, note that since the driving L\'{e}vy process is a Brownian motion, it holds per construction that the linear innovations $(\epsilon_k)_{k \in \Z}$ of the process $(Y(kh))_{k \in \Z}$ are i.i.d. $\NN(0, V)$--distributed (cf. \Cref{definition innovations}).
% $\epsilon_{\vt_0, k} \stackrel{\text{i.i.d.}}{\sim} \NN(0, V_{\vt_0})$.
Moreover, per assumption it also holds that $\epsilon_{\vt^\ast, k} = \epsilon_{k}$ for every $k \in \Z$, hence we also have that $\epsilon_{\vt^\ast, k} \stackrel{}{\sim} \NN(0, V_{\vt^\ast})$ and $V_{\vt^\ast} = V$.
By definition
$$\II(\vt^\ast) = \lim_{n \to \infty} n \Var \left( \nabla_\vt \LL (\vt^\ast, Y^n) \right), $$
which means that for $i, j \in \{ 1, \ldots, N(\Theta)\}$ we have to study terms of the form
\begin{eqnarray*}
\Var \left( n\nabla_\vt \LL (\vt^\ast, Y^n) \right)_{ij}%{\Var \left( \sum_{k=1}^n \nabla_\vt l_{\vt^\ast,k} \right)_{ij}}
\hspace*{-0.3cm}
 &\stackrel{\eqref{partdevl}}{=}&\hspace*{-0.3cm}
 %\E \left[ \sum_{k=1}^n \sum_{l=1}^{n} \left( \tr \left( V^{-1}_{\vt^\ast} \left( I_d - \epsilon_{\vt^\ast, k} \epsilon^T_{\vt^\ast, k} V^{-1}_{\vt^\ast} \right) \partial_i V_{\vt^\ast} \right) + 2 \partial_i \epsilon^T_{\vt^\ast, k} V^{-1}_{\vt^\ast} \epsilon_{\vt^\ast, k} \right) \right. \notag \\
%&&\left. \cdot \left( \tr \left( V^{-1}_{\vt^\ast} \left( I_d - \epsilon_{\vt^\ast, l} \epsilon^T_{\vt^\ast, l} V^{-1}_{\vt^\ast} \right) \partial_j V_{\vt^\ast} \right) + 2 \partial_j \epsilon^T_{\vt^\ast, l} V^{-1}_{\vt^\ast} \epsilon_{\vt^\ast, l} \right) \right] \notag \\
\sum_{k=1}^n \E \left[ \left( \tr \left( V^{-1}_{\vt^\ast} \partial_i V_{\vt^\ast} \right) - \tr \left( V^{-1}_{\vt^\ast} \epsilon_{\vt^\ast,k } \epsilon^T_{\vt^\ast,k } V^{-1}_{\vt^\ast} \partial_i V_{\vt^\ast} \right) + 2 \partial_i \epsilon^T_{\vt_\ast,k} V^{-1}_{\vt^\ast} \epsilon_{\vt^\ast,k} \right) \right. \notag \\
&&\hspace*{0.6cm}\left. \cdot \left( \tr \left( V^{-1}_{\vt^\ast} \partial_j V_{\vt^\ast} \right) - \tr \left( V^{-1}_{\vt^\ast} \epsilon_{\vt^\ast,k } \epsilon^T_{\vt^\ast,k } V^{-1}_{\vt^\ast} \partial_j V_{\vt^\ast} \right) + 2 \partial_j \epsilon^T_{\vt_\ast,k} V^{-1}_{\vt^\ast} \epsilon_{\vt^\ast,k} \right) \right] \notag \\
&& \hspace*{-0.6cm}+ \sum_{k=1}^n \sum_{\substack{ l =1 \\ l \neq k}}^n \E \left[ \left( \tr \left( V^{-1}_{\vt^\ast} \partial_i V_{\vt^\ast} \right) - \tr \left( V^{-1}_{\vt^\ast} \epsilon_{\vt^\ast,k } \epsilon^T_{\vt^\ast,k } V^{-1}_{\vt^\ast} \partial_i V_{\vt^\ast} \right) + 2 \partial_i \epsilon^T_{\vt_\ast,k} V^{-1}_{\vt^\ast} \epsilon_{\vt^\ast,k} \right) \right. \notag \\
&&\hspace*{0.6cm}\left. \cdot \left( \tr \left( V^{-1}_{\vt^\ast} \partial_j V_{\vt^\ast} \right) - \tr \left( V^{-1}_{\vt^\ast} \epsilon_{\vt^\ast,l } \epsilon^T_{\vt^\ast,l } V^{-1}_{\vt^\ast} \partial_j V_{\vt^\ast} \right) + 2 \partial_j \epsilon^T_{\vt_\ast,l} V^{-1}_{\vt^\ast} \epsilon_{\vt^\ast,l} \right)   \right] \notag\\
&=:&\sum_{k=1}^na_k+\sum_{k=1}^n \sum_{\substack{ l =1 \\ l \neq k}}^nb_{k,l}. \label{BMcalc2}
\end{eqnarray*}
We start to investigate $a_k$. By definition, every innovation $\epsilon_{\vt^\ast, k}$ is orthogonal to $\overline{\text{span}} \{ Y(jh): - \infty < j < k \}$ and by \Cref{serieslemma}(b) both $\partial_i \epsilon_{\vt^\ast, k}$ and $\partial_j \epsilon_{\vt^\ast, k}$ are elements of $\overline{\text{span}} \{ Y(jh): - \infty < j < k \}$.
Hence, $\epsilon_{\vt^\ast, k}$ is independent of $\partial_i \epsilon_{\vt^\ast, k}$ and $\partial_j \epsilon_{\vt^\ast, k}$.
This, together with the independence of the innovation sequence $(\epsilon_{\vt^\ast, k})_{k\in\N}$, the fact that $\E[ \partial_i \epsilon_{\vt^\ast,k} ] = 0$, $\E[\epsilon_{\vt^\ast,k} \epsilon^T_{\vt^\ast,k}] = V_{\vt^\ast}$ and the interchangeability of trace and expectation, allows us to simplify
\begin{align}
 a_k&= -\tr \left( V^{-1}_{\vt^\ast} \partial_i V_{\vt^\ast} \right) \tr \left( V^{-1}_{\vt^\ast} \partial_j V_{\vt^\ast} \right)% \notag \\
 + \E \left[ \tr \left( V^{-1}_{\vt^\ast} \epsilon_{\vt^\ast,k } \epsilon^T_{\vt^\ast,k } V^{-1}_{\vt^\ast} \partial_i V_{\vt^\ast} \right) \tr \left( V^{-1}_{\vt^\ast} \epsilon_{\vt^\ast,k } \epsilon^T_{\vt^\ast,k } V^{-1}_{\vt^\ast} \partial_j V_{\vt^\ast} \right) \right] \notag \\
%& -2 \E \left[  \partial_j \epsilon^T_{\vt^\ast,k} V^{-1}_{\vt^\ast} \epsilon_{\vt^\ast,k} \tr \left( V^{-1}_{\vt^\ast} \epsilon_{\vt^\ast,k } \epsilon^T_{\vt^\ast,k } V^{-1}_{\vt^\ast} \partial_i V_{\vt^\ast} \right) \right] \notag \\
%&- 2 \E \left[ \partial_i \epsilon^T_{\vt^\ast,k } V^{-1}_{\vt^\ast} \epsilon_{\vt^\ast,k} \tr \left( V^{-1}_{\vt^\ast} \epsilon_{\vt^\ast,k} \epsilon^T_{\vt^\ast,k} V^{-1}_{\vt^\ast} \partial_j V_{\vt^\ast} \right) \right]
&\quad + 4 \E \left[ \partial_i \epsilon^T_{\vt^\ast,k} V^{-1}_{\vt^\ast} \epsilon_{\vt^\ast,k} \partial_j \epsilon^T_{\vt^\ast,k } V^{-1}_{\vt^\ast} \epsilon_{\vt^\ast,k} \right] \notag\\
&=:a_k^{(1)}+a_k^{(2)}+a_k^{(3)}. \label{BMcalc3}
\end{align}
For the second term, we define $\tilde{\epsilon}_{\vt^\ast,k} = V^{-\frac12}_{\vt^\ast} \epsilon_{\vt^\ast,k} \stackrel{\text{}}{\sim} \NN(0, I_{d \times d})$ and have by standard calculation rules for Kronecker products (\cite[Proposition 7.1.6 and Proposition 7.1.12]{bernstein}):
\begin{eqnarray*}
%&\E \left[ \tr \left( V^{-1}_{\vt^\ast} \epsilon_{\vt^\ast,k } \epsilon^T_{\vt^\ast,k } V^{-1}_{\vt^\ast} \partial_i V_{\vt^\ast} \right) \tr \left( V^{-1}_{\vt^\ast} \epsilon_{\vt^\ast,k } \epsilon^T_{\vt^\ast,k } V^{-1}_{\vt^\ast} \partial_j V_{\vt^\ast} \right) \right] \notag \\
a_k^{(2)}&=& \E \left[ \tr \left( \left( V^{-\frac12}_{\vt^\ast} \tilde{\epsilon}_{\vt^\ast,k} \tilde{\epsilon}^T_{\vt^\ast,k} V^{-\frac12}_{\vt^\ast} \partial_i V_{\vt^\ast} \right) \otimes \left( V^{-\frac12}_{\vt^\ast} \tilde{\epsilon}_{\vt^\ast,k} \tilde{\epsilon}^T_{\vt^\ast,k} V^{-\frac12}_{\vt^\ast} \partial_j V_{\vt^\ast} \right) \right) \right] \notag \\
%&= \E \left[ \tr \left( \left( \left( V^{-\frac12}_{\vt^\ast} \tilde{\epsilon}_{\vt^\ast,k} \tilde{\epsilon}^T_{\vt^\ast,k} V^{-\frac12}_{\vt^\ast}  \right) \otimes \left( V^{-\frac12}_{\vt^\ast} \tilde{\epsilon}_{\vt^\ast,k} \tilde{\epsilon}^T_{\vt^\ast,k} V^{-\frac12}_{\vt^\ast} \right) \right) \cdot \left( \partial_i V_{\vt^\ast} \otimes \partial_j V_{\vt^\ast} \right) \right) \right] \notag \\
%&= \E \left[ \tr \left( \left( V^{-\frac12}_{\vt^\ast} \otimes V^{-\frac12}_{\vt^\ast} \right) \cdot \left( \tilde{\epsilon}_{\vt^\ast,k} \tilde{\epsilon}^T_{\vt^\ast,k} \otimes \tilde{\epsilon}_{\vt^\ast,k} \tilde{\epsilon}^T_{\vt^\ast,k} \right) \cdot \left( V^{-\frac12}_{\vt^\ast} \partial_i V_{\vt^\ast} \otimes V^{-\frac12}_{\vt^\ast} \partial_j V_{\vt^\ast} \right) \right) \right] \notag \\
&=& \tr \left( \left( V^{-\frac12}_{\vt^\ast} \otimes V^{-\frac12}_{\vt^\ast} \right) \cdot \E \left[ \tilde{\epsilon}_{\vt^\ast,k} \tilde{\epsilon}^T_{\vt^\ast,k} \otimes \tilde{\epsilon}_{\vt^\ast,k} \tilde{\epsilon}^T_{\vt^\ast,k} \right] \cdot \left( V^{-\frac12}_{\vt^\ast} \partial_i V_{\vt^\ast} \otimes V^{-\frac12}_{\vt^\ast} \partial_j V_{\vt^\ast} \right) \right). \label{BMcalc4}
\end{eqnarray*}
Since $\tilde{\epsilon}_{\vt^\ast,k} \sim \NN(0, I_{d \times d})$, by means of \cite[Theorem 1]{balestra} the expectation appearing in the last line is
$$\E \left[ \tilde{\epsilon}_{\vt^\ast,k} \tilde{\epsilon}^T_{\vt^\ast,k} \otimes \tilde{\epsilon}_{\vt^\ast,k} \tilde{\epsilon}^T_{\vt^\ast,k} \right]  = K_{d,d} + I_{d^2 \times d^2} + \vecc(I_{d \times d}) \otimes \vecc( I_{d \times d} )^T,$$
where $K_{d,d}$ is the $d^2 \times d^2$ Kronecker permutation matrix (\cite[Eq. (7.1.20)]{bernstein}).
Together with the linearity and the cyclic permutation property of the trace, we use this to obtain
\begin{align*}
 a_k^{(2)}&= \tr \left( K_{d,d} \left( V^{-\frac12}_{\vt^\ast} \otimes V^{-\frac12}_{\vt^\ast} \right) \left( V^{-\frac12}_{\vt^\ast} \partial_i V_{\vt^\ast} \otimes V^{-\frac12}_{\vt^\ast} \partial_j V_{\vt^\ast} \right) \right) \notag \\
&\quad+ \tr \left( \left( V^{-\frac12}_{\vt^\ast} \otimes V^{-\frac12}_{\vt^\ast} \right) \left( V^{-\frac12}_{\vt^\ast} \partial_i V_{\vt^\ast} \otimes V^{-\frac12}_{\vt^\ast} \partial_j V_{\vt^\ast} \right) \right) \notag \\
&\quad+ \tr \left( (\vecc(I_{d \times d}) \otimes \vecc( I_{d \times d} )^T) \left( V^{-\frac12}_{\vt^\ast} \otimes V^{-\frac12}_{\vt^\ast} \right) \left( V^{-\frac12}_{\vt^\ast} \partial_i V_{\vt^\ast} \otimes V^{-\frac12}_{\vt^\ast} \partial_j V_{\vt^\ast} \right) \right) \notag \\
&= \tr \left( K_{d,d} \left( V^{-1}_{\vt^\ast} \partial_i V_{\vt^\ast} \otimes V^{-1}_{\vt^\ast} \partial_j V_{\vt^\ast} \right) \right) + \tr \left( V^{-1}_{\vt^\ast} \partial_i V_{\vt^\ast} \otimes V^{-1}_{\vt^\ast} \partial_j V_{\vt^\ast} \right) \notag \\
&\quad+ \tr \left( (\vecc(I_{d \times d}) \otimes \vecc( I_{d \times d} )^T) (V^{-1}_{\vt^\ast} \partial_i V_{\vt^\ast} \otimes V^{-1}_{\vt^\ast} \partial_j V_{\vt^\ast} ) \right). \label{BMcalc5}
\end{align*}
We now apply \Cref{tracelemma}  as well as \cite[Fact 7.4.30 xviii) and Proposition 7.1.12]{bernstein} to get
$$a_k^{(2)} = 2 \tr \left( V^{-1}_{\vt^\ast} \partial_i V_{\vt^\ast}  V^{-1}_{\vt^\ast} \partial_j V_{\vt^\ast} \right) + \tr \left( V^{-1}_{\vt^\ast} \partial_i V_{\vt^\ast} \right) \tr \left( V^{-1}_{\vt^\ast} \partial_j V_{\vt^\ast} \right).$$
It remains to consider $a_k^{(3)}$ in \eqref{BMcalc3}. The independence  of $\partial_j \epsilon_{\vt^\ast, k} \partial_i \epsilon^T_{\vt^\ast, k}$
 and $\epsilon_{\vt^\ast,k}$,  the cyclic permutation property of the trace and the interchangeability of expectation and trace leads to
\begin{eqnarray*}
a_k^{(3)}
%&=&  \E \left[ \tr \left( \partial_i \epsilon^T_{\vt^\ast,k} V^{-1}_{\vt^\ast} \epsilon_{\vt^\ast,k} \epsilon^T_{\vt^\ast,k} V^{-1}_{\vt^\ast}  \partial_j \epsilon_{\vt^\ast,k }  \right) \right] \\
&=& \E \left[ \tr \left( V^{-1}_{\vt^\ast} \epsilon_{\vt^\ast,k} \epsilon^T_{\vt^\ast,k} V^{-1}_{\vt^\ast}  \partial_j \epsilon_{\vt^\ast,k } \partial_i \epsilon^T_{\vt^\ast,k} \right) \right]\\
&=& \tr \left( \E \left[ V^{-1}_{\vt^\ast} \epsilon_{\vt^\ast,k} \epsilon^T_{\vt^\ast,k} V^{-1}_{\vt^\ast} \right] \E \left[ \partial_j \epsilon_{\vt^\ast,k } \partial_i \epsilon^T_{\vt^\ast,k} \right] \right) \\
&=& \tr \left( V^{-1}_{\vt^\ast} \E \left[ \partial_j \epsilon_{\vt^\ast,k } \partial_i \epsilon^T_{\vt^\ast,k} \right] \right) \\
&=& \E \left[ \partial_i \epsilon^T_{\vt^\ast,k} V^{-1}_{\vt^\ast} \partial_j \epsilon_{\vt^\ast,k } \right].
\end{eqnarray*}
Combining those calculations finally results in
$$a_k =a_k^{(1)}+a_k^{(2)}+a_k^{(3)}= 2 \tr \left( V^{-1}_{\vt^\ast} \partial_i V_{\vt^\ast}  V^{-1}_{\vt^\ast} \partial_j V_{\vt^\ast} \right) + 4 \E \left[ \partial_i \epsilon^T_{\vt^\ast,k} V^{-1}_{\vt^\ast} \partial_j \epsilon_{\vt^\ast,k } \right].$$
By similar calculations, we can verify that  $b_{k,l}=0$ for $k \neq l$. \\
%\begin{align*}
%&\E \left[ \left( \tr \left( V^{-1}_{\vt^\ast} \partial_i V_{\vt^\ast} \right) - \tr \left( V^{-1}_{\vt^\ast} \epsilon_{\vt^\ast,k } \epsilon^T_{\vt^\ast,k } V^{-1}_{\vt^\ast} %\partial_i V_{\vt^\ast} \right) + 2 \partial_i \epsilon^T_{\vt_\ast,k} V^{-1}_{\vt^\ast} \epsilon_{\vt^\ast,k} \right) \right. \notag \\
%&\left. \cdot \left( \tr \left( V^{-1}_{\vt^\ast} \partial_i V_{\vt^\ast} \right) - \tr \left( V^{-1}_{\vt^\ast} \epsilon_{\vt^\ast,l } \epsilon^T_{\vt^\ast,l } V^{-1}_{\vt^\ast} %\partial_j V_{\vt^\ast} \right) + 2 \partial_j \epsilon^T_{\vt_\ast,l} V^{-1}_{\vt^\ast} \epsilon_{\vt^\ast,l} \right)   \right] = 0,
%\end{align*}
Finally, this implies
$$(\II(\vt^\ast))_{ij}=a_k = 2 \tr \left( V^{-1}_{\vt^\ast} \partial_i V_{\vt^\ast}  V^{-1}_{\vt^\ast} \partial_j V_{\vt^\ast} \right) + 4 \E \left[ \partial_i \epsilon^T_{\vt^\ast,k} V^{-1}_{\vt^\ast} \partial_j \epsilon_{\vt^\ast,k } \right].$$
By \cite[(2.33a) and (2.33b)]{schlemmstelzer}, this term is equal to $(2 \HH(\vt^\ast))_{ij}$ as proclaimed.\\
(b) \, Denote by $v_1,\ldots,v_{N(\Theta_0^E)}$ the eigenvectors of $\HH(\vt_0^E)$ which are an
orthonormal basis of $\R^{N(\Theta_0^E)}$. Define $F=(v_1,\ldots,v_{N(\Theta_0^E)-1})\in\R^{N(\Theta_0^E)\times(N(\Theta_0^E)-1)}$
and let ${\Theta}_0 \subseteq F^T\Theta_0^E$ be compact such that $F {\Theta}_0+(\vt_0^E-FF^T\vt_0^E)\subseteq \Theta_0^E$ and $F^T\vt_0^E\in {\Theta}_0$. Define
$$(A_{\vt}, B_{\vt}, C_{\vt}, L_{\vt})_{\vt \in {\Theta}_0}:=(A_{F\vt+(\vt_0^E-FF^T\vt_0^E)}, B_{F\vt+(\vt_0^E-FF^T\vt_0^E)}, C_{F\vt+(\vt_0^E-FF^T\vt_0^E)}, L_{F\vt+(\vt_0^E-FF^T\vt_0^E)})_{\vt \in {\Theta}_0}.$$
Then $\vt_0=F^T\vt_0^E$, ${\Theta}_0$ is nested in $\Theta_0^E$ with map $F$ and  satisfies \autoref{as_D}, and $N({\Theta}_0)=N(\Theta_0^E)-1$.
Moreover, the eigenvectors $v_1,\ldots,v_{N(\Theta_0^E)-1}$ are basis vectors of the image of $F$ and $v_{N(\Theta_0^E)}$
is a basis of the kernel of $F^T$.  Then $v_{N(\Theta_0^E)}$ is an eigenvector of $\MM_F(\vt_0^E)^{\frac12} \II(\vt^E_0) \MM_F(\vt^E_0)^{\frac12}$ for the eigenvalue $2$
and $v_1,\ldots,v_{N(\Theta_0^E)-1}$  are eigenvectors of $\MM_F(\vt^E_0)^{\frac12} \II(\vt^E_0) \MM_F(\vt^E_0)^{\frac12}$ for the eigenvalue $0$ as well.
\end{proof}

{\small\bibliography{IC_CARMA_2015_05_01_arxiv}}
\bibliographystyle{apt}
\end{document}